\documentclass[12pt]{amsart}
\usepackage[utf8]{inputenc}
\usepackage{amssymb}
\usepackage{hyperref}
\usepackage[final]{showkeys}

\input xy
\xyoption{all}

\theoremstyle{definition}
\newtheorem{mydef}{Definition}[section]
\newtheorem{lem}[mydef]{Lemma}
\newtheorem{thm}[mydef]{Theorem}
\newtheorem{cor}[mydef]{Corollary}

\newtheorem{question}[mydef]{Question}

\newtheorem{prop}[mydef]{Proposition}
\newtheorem{defin}[mydef]{Definition}

\newtheorem{remark}[mydef]{Remark}
\newtheorem{notation}[mydef]{Notation}
\newtheorem{fact}[mydef]{Fact}

\newcommand{\fct}[2]{{}^{#1}#2}

\newcommand{\ba}{\bar{a}}
\newcommand{\bb}{\bar{b}}
\newcommand{\bc}{\bar{c}}

\newcommand{\bigN}{\widehat{N}}
\newcommand{\bigf}{\widehat{f}}

\newcommand{\dom}[1]{\text{dom}(#1)}

\newcommand{\cf}[1]{\text{cf} (#1)}
\newcommand{\seq}[1]{\langle #1 \rangle}
\newcommand{\rest}{\upharpoonright}

\newcommand{\s}{\mathfrak{s}}
\newcommand{\ts}{\mathfrak{t}}

\newcommand{\id}{\text{id}}

\newcommand{\hatr}{\widehat{r}}

\def\lta{\precneqq}
\def\lea{\prec}
\def\gta{\succneqq}
\def\gea{\succ}

\newbox\noforkbox \newdimen\forklinewidth
\forklinewidth=0.3pt \setbox0\hbox{$\textstyle\smile$}
\setbox1\hbox to \wd0{\hfil\vrule width \forklinewidth depth-2pt
 height 10pt \hfil}
\wd1=0 cm \setbox\noforkbox\hbox{\lower 2pt\box1\lower
2pt\box0\relax}
\def\unionstick{\mathop{\copy\noforkbox}\limits}
\newcommand{\nf}{\unionstick}
\newcommand{\nfs}[4]{#2 \nf_{#1}^{#4} #3}

\newcommand{\tp}{\text{tp}}
\newcommand{\gtp}{\text{gtp}}
\newcommand{\Ss}{S}
\newcommand{\Sna}{S^\text{na}}
\newcommand{\Sbs}{S^\text{bs}}
\newcommand{\Sana}[1]{S^{#1,\text{na}}}
\newcommand{\Sabs}[1]{S^{#1,\text{bs}}}

\newcommand{\F}{\mathcal{F}}
\newcommand{\LS}{\text{LS}}
\newcommand{\OR}{\text{OR}}
\newcommand{\goodm}{\text{good}^-}
\newcommand{\goodms}[1]{\text{good}^{-#1}}

\title{Tameness and frames revisited}
\date{\today} 

\author{Will Boney}
\email{wboney@math.harvard.edu}
\urladdr{http://math.harvard.edu/\textasciitilde wboney/}
\address{Mathematics Department, Harvard University, Cambridge, Massachusetts, USA}

\author{Sebastien Vasey}
\email{sebv@cmu.edu}
\urladdr{http://math.cmu.edu/\textasciitilde svasey/}
\address{Department of Mathematical Sciences, Carnegie Mellon University, Pittsburgh, Pennsylvania, USA}

\thanks{This material is based upon work done while the first author was supported by the National Science Foundation under Grant No. DMS-1402191 and the second author was supported by the Swiss National Science Foundation under Grant No.\ 155136.}

\begin{document}

\begin{abstract}
We study the problem of extending an abstract independence notion for types of singletons (what Shelah calls a good frame) to longer types. Working in the framework of tame abstract elementary classes, we show that good frames can always be extended to types of independent sequences. As an application, we show that tameness and a good frame imply Shelah's notion of dimension is well-behaved, complementing previous work of Jarden and Sitton. We also improve a result of the first author on extending a frame to larger models. 
\end{abstract}

\maketitle

\tableofcontents

\section{Introduction}

Good $\lambda$-frames are an axiomatic notion of independence in abstract elementary classes (AECs) introduced by Shelah \cite[Chapter II]{shelahaecbook}. They are one of the main tools in the classification theory of AECs. They describe a relation ``$p$ does not fork over $M$'' for certain types of singletons over models of size $\lambda$. The frame's nonforking relation is required to satisfy properties akin to those of forking in a first-order superstable theory. The definition can be generalized to that of a good $(<\alpha, [\lambda, \theta))$-frame, where instead of types of singletons one allows types of sequences of less than $\alpha$-many elements, and instead of the models being of size $\lambda$, one allows their size to lie in the interval $[\lambda, \theta)$.

    There are at least two questions one can ask about frames: first, under what hypotheses do they exist? Second, can we extend them? That is, assuming there is a frame can we extend it to give a nonforking definition for larger models or longer\footnote{The length of a type is the length or indexing set of a tuple that satisfies it.  See Section \ref{galdef-sec} for a definition.} types?

    Shelah tackles these problems in \cite[Chapters II and III]{shelahaecbook}, but the answers use strong model-theoretic hypotheses (typically categoricity in two successive cardinals $\lambda$ and $\lambda^+$ together with few models in $\lambda^{++}$), as well as set-theoretic hypotheses (like the weak generalized continuum hypothesis, $2^\lambda < 2^{\lambda^+}$)\footnote{Shelah also looks at the existence problem in a more global setup and in ZFC in \cite[Chapter IV]{shelahaecbook} but does not study the extension problem there.}.

    Recently, the two questions above have been studied in the framework of \emph{tame} AECs. Tameness is a locality property of AECs isolated by Grossberg and VanDieren \cite{tamenessone} from an argument in \cite{sh394}. Grossberg and VanDieren have shown \cite{tamenesstwo, tamenessthree} that Shelah's eventual categoricity conjecture from a successor holds in tame AECs, and the first author \cite{tamelc-jsl} (building on work of Makkai-Shelah \cite{makkaishelah}) has shown that tameness follows from a large cardinal axiom. Many examples of interest are also known to be tame.

    Under tameness, the second author has shown that frames exist in ZFC assuming a reasonable categoricity hypothesis \cite{ss-tame-jsl}, and the first author has shown \cite{ext-frame-jml} that frames can be extended to larger \emph{models} under the assumption of tameness for types of length two. In this paper, we further study the frame extension question in tame AECs. We look at the problem of \emph{elongating} the frame: extending it to longer \emph{types}. 

Let us give discuss a natural approach to the problem and its shortcomings. In stable first-order theories, we have that $ab \nf_A B$ if and only if $a \nf_{A} B$ and $b \nf_{Aa} Ba$. One might think that this allows us to define forking for types of all lengths if we have a definition of forking for singleton types (as in \cite{dep-pregeom}). However, this turns out not to work in full generality, as good frames only define forking over \emph{models}\footnote{Note that an example of Shelah (see \cite[Section 4]{hyttinen-lessmann}) shows that there exists a superstable homogeneous diagram where extension (over sets) fails for any reasonable independence notion.}. We might want to say that $ab \nf_{M} N$ if and only if there are $M \prec M' \prec N'$ with $N \prec N'$ and $a \in M'$ such that $a \nf_M N$ and $b \nf_{M'} N'$.  This means that a choice must be made for the models $N'$ and, especially, $M'$ and this choice can cause problems.  In particular, if $M'$ is too big, then uniqueness of nonforking extensions can fail.  This does not cause issues in the first-order context essentially because there is a prime/minimal set containing $A$ and $a$ (namely $Aa$).

There are two options to work around this issue.  The first option is to assume the existence of a \emph{unique} prime/minimal extension of $Ma$; Shelah says that the frame is \emph{weakly successful} \cite[Definition III.1.1]{shelahaecbook} if this is the case. Shelah proved \cite[Section II.6]{shelahaecbook} that weakly successful frames can be elongated as desired without any assumption of tameness. Shelah has also shown \cite[II.5]{shelahaecbook} that a good $\lambda$-frame is weakly successful when the underlying AEC has few models in $\lambda^{++}$ and certain set-theoretic hypotheses hold. It is not known whether being weakly successful follows from tameness\footnote{After the initial submission of this paper, it has been shown that being weakly successful follows from a stronger locality property: full tameness and shortness \cite[Section 11]{indep-aec-apal}.}.

The second option is to strengthen the condition on nonforking, essentially setting $ab \nf_A B$ if and only if $a \nf_A B$ and $b \nf_A Ba$ (the noncanonical choice of a cover for $Ba$ is less important).  This loses some information about nonforking, so only works for certain kinds of types: types of independent sequences. As we show, this has the advantage of working in a larger class of AECs, i.e.\ those with frames that are not weakly successful, although we do assume tameness to prove the symmetry property.

This brings us to the precise statement of the main result of this paper.
    
\begin{thm}\label{main-thm}
Let $K$ be an AEC with amalgamation and $\F = [\lambda, \theta)$ be an interval of cardinals.
\begin{enumerate}
	\item Assume $\s$ satisfies the axioms of a good $\F$-frame, except possibly symmetry. Then $\s$ can be extended to a certain frame $\s^{<\theta}$ which satisfies the axioms of a good $(<\theta, \F)$-frame, except possibly for symmetry.
	\item If $K$ is $\lambda$-tame, then both $\s$ and $\s^{<\theta}$ also satisfy symmetry.
\end{enumerate}
\end{thm}

The following two questions are open.
\begin{question} \
\begin{itemize}
	\item If $\s$ is a good frame in a tame AEC, must $\s$ be weakly successful?
	\item Is there an example of a good $\lambda$-frame (necessarily not weakly successful) that has no better extension to longer types than independent sequences?
\end{itemize}
\end{question}

As has already been alluded to, in Theorem \ref{main-thm} the frame $\s$ is elongated by use of \emph{independent sequences} (see Definition \ref{indep-seq-def} here, or Shelah \cite[Definition III.5.2]{shelahaecbook}).  Independent sequences in that context have been previously studied by Shelah \cite[III.5]{shelahaecbook} and Jarden and Sitton \cite{jasi}. Throughout these studies, several additional assumptions have appeared--such as $\s$ being weakly successful or having continuity of serial independence\footnote{This and other variations on continuity are defined and explored in Section \ref{continuity-subsec}.}--that we are able to eliminate or replace with the hypothesis of tameness.

We present two applications of Theorem \ref{main-thm}.  The first involves a natural notion of dimension that Shelah introduced with the goal of building a theory of regular types for AECs \cite[Definition III.5.12]{shelahaecbook}: Let us define the dimension of a type $p$ in an ambient model $N$, $\dim (p, N)$ to be the size of a maximal independent set of realizations of $p$ in $N$.  In the first-order case, Shelah \cite[III.4.21.(2)]{shelahfobook} shows that, under stability, every infinite maximal independent set of realizations of $p$ has the same size.  In the AEC framework, Shelah \cite[III.5.14]{shelahaecbook} first showed that this held when the frame is weakly successful, and Jarden and Sitton \cite{jasi} have refined these hypotheses. The analysis of this paper allows us to show that the dimension is well-behaved in any tame AEC with a good frame (see Corollary \ref{good-dimension} and the surrounding discussion). This gives a natural nonelementary framework in which a theory of regular types could be studied.

The second application involves the project of extending a frame to larger \emph{models} using tameness. As mentioned above, the first author has shown that this is possible if one assumes tameness for types of length two. Analyzing the elongations of frames allows us to give an aesthetic improvement: we remove this strange assumption and replace it with only tameness for types of length one (see Corollary \ref{will-transfer-improved} and the preceding discussions). While no example of an AEC that is tame for types of length one and not for length two is known, thinking about this statement led us to the main theorem of this paper. Further, we are told that Rami Grossberg conjectured Corollary \ref{will-transfer-improved} already in 2006 (he told it to Adi Jarden and John Baldwin); our result proves Grossberg's original conjecture.

Since this paper was first circulated (June 2014), several applications of Corollary \ref{will-transfer-improved} have been found. They include Shelah's eventual categoricity conjecture for universal classes \cite{ap-universal-v10, categ-universal-2-v2-toappear}, as well as a downward categoricity transfer for tame AECs \cite{downward-categ-tame-apal} (the latter actually uses the theory of independent sequences in good frames developped in Section \ref{indep-seq-good}). In \cite{counterexample-frame-v2}, the authors show that a natural good frame appearing in the Hart-Shelah example is not weakly successful, and in \cite{quasimin-aec-v3} the second author studies an example of Shelah where a good frame cannot be extended to all types. These examples show that this paper \emph{strictly} generalizes Shelah's study of independent sequences in \cite[Section III.5]{shelahaecbook}.

The paper is structured as follows. In Section \ref{prelim}, we review background in the theory of AECs. In Section \ref{good-frame}, we give the definition of good frames and prove some easy general facts. In Section \ref{indep-seq-good}, we define independent sequences and show how to use them to extend a frame for types of singletons to a frame for longer types. We show all properties are preserved in the process, except perhaps symmetry. In Section \ref{sym-long-frames}, we give conditions under which symmetry also transfers and show how to use it to define a well-behaved notion of dimension. In Section \ref{going-up}, we prove the promised applications to dimension and tameness.

This paper was written while the authors were working on Ph.D.\ theses under the direction of Rami Grossberg at Carnegie Mellon University and they would like to thank Professor Grossberg for his guidance and assistance in our research in general and in this work specifically. The authors would also like to thank the referee for their helpful report that greatly assisted the clarity and presentation of this paper.

\section{Preliminaries}\label{prelim}

\subsection{Abstract elementary classes}

We assume the reader is familiar with the definition of an abstract elementary class (AEC) and the basic related concepts. See Grossberg's \cite{grossberg2002} or Baldwin's \cite{baldwinbook09} for an introduction to AECs.  A more advanced introduction to frames can be found in \cite[Chapter II]{shelahaecbook} 

For the rest of this section, fix an AEC $K$. We denote the partial ordering on $K$ by $\lea$, and write $M \lta N$ if $M \lea N$ and $M \neq N$. 

For $K$ an abstract elementary class and $\F$ an interval\footnote{The definitions that follow make sense for an arbitrary set of cardinals $\F$, but the proofs of most of the facts below require that $\F$ is an interval.} of cardinals of the form $[\lambda, \theta)$, where $\theta > \lambda \ge \LS (K)$ is either a cardinal or $\infty$, let $K_\F := \{M \in K \mid \|M\| \in \F\}$. We write $K_\lambda$ instead of $K_{\{\lambda\}}$, $K_{\ge \lambda}$ instead of $K_{[\lambda, \infty)}$ and $K_{\le \lambda}$ instead of $K_{[\LS (K), \lambda]}$.

The following properties of AECs are classical:

\begin{defin}
  Let $\F$ be an interval of cardinals as above.
  \begin{enumerate}
    \item $K_{\F}$ has \emph{amalgamation} if for any $M_0 \lea M_\ell \in K_{\F}$, $\ell = 1,2$ there exists $N \in K_{\F}$ and $f_\ell : M_\ell \xrightarrow[M_0]{} N$, $\ell = 1,2$.
    \item $K_{\F}$ has \emph{joint embedding} if for any $M_\ell \in K_{\F}$, $\ell = 1,2$ there exists $N \in K_{\F}$ and $f_\ell : M_\ell \rightarrow N$, $\ell = 1,2$.
    \item $K_{\F}$ has \emph{no maximal models} if for any $M \in K_{\F}$ there exists $N \gta M$ in $K_{\F}$.
  \end{enumerate}
\end{defin}




\subsection{Galois types, stability, and tameness}\label{galdef-sec}

We assume familiarity with Galois types (see \cite[Section 6]{grossberg2002}). For $M \in K$ and $\alpha$ an ordinal, we write $\Ss^\alpha (M)$ for the set of Galois types of sequences of length $\alpha$ over $M$. We write $\Ss^{<\alpha} (M)$ for $\bigcup_{\beta < \alpha} \Ss^{\beta} (M)$ and $\Ss^{<\infty} (M)$ for $\bigcup_{\beta \in \OR} \Ss^\beta (M)$. We write $\Ss (M)$ for $\Ss^1 (M)$ and $\Sna (M)$ for the set of \emph{nonalgebraic} 1-types over $M$, that is:

$$
\Sna (M) := \{\gtp (a / M; N) \mid a \in N \backslash M, M \lea N \in K\}
$$

From now on, we will write $\tp (a / M; N)$ for $\gtp (a / M; N)$. If $p \in \Ss^{\alpha} (M)$, we define $\ell (p) := \alpha$ and $\dom{p} := M$. Note that $\alpha$ is an invariant of the Galois type and is referred to as its \emph{length}.

Say $p = \tp (\ba / M; N) \in \Ss^\alpha (M)$, where $\ba = \seq{a_i :i < \alpha}$. For $X \subseteq \alpha$ and $M_0 \lea M$, write $p^X \upharpoonright M_0$ for $\tp (\ba_X / M_0; N)$, where $\ba_X := \seq{a_i : i \in X}$. We say $p \in \Sana{\alpha} (M)$ if $a_i \notin M$ for all $i < \alpha$, and similarly define $\Sana{<\alpha} (M)$ (it is easy to check these definitions do not depend on the choice of $\ba$ and $N$).

We briefly review the notion of tameness. Although it appears implicitly (for saturated models) in Shelah \cite{sh394}, tameness as a property of AECs was first introduced in Grossberg and VanDieren \cite{tamenessone} and used to prove a stability spectrum theorem there.

\begin{defin}[Tameness]\label{tameness-def}
  Let $\theta > \lambda \ge \LS (K)$ and let $\mathcal{G} \subseteq \bigcup_{M \in K} \Ss^{<\infty} (M)$ be a family of types. We say that $K$ is \emph{$(\lambda, \theta)$-tame for $\mathcal{G}$} if for any $M \in K_{\le \theta}$ and any $p, q \in \mathcal{G} \cap \Ss^{<\infty} (M)$, if $p \neq q$, then there exists $M_0 \lea M$ of size $\le \lambda$ such that $p \upharpoonright M_0 \neq q \upharpoonright M_0$. We define similarly $(\lambda, <\theta)$-tame, $(<\lambda, \theta)$-tame, etc. When $\theta = \infty$, we omit it. $(\lambda, \theta)$-tame for $\alpha$-types means $(\lambda, \theta)$-tame for $\bigcup_{M \in K} \Ss^\alpha (M)$, and similarly for $<\alpha$-types. When $\alpha = 1$, we omit it and simply say $(\lambda, \theta)$-tame.
\end{defin}

We also recall that we can define a notion of stability:

\begin{defin}[Stability]
  Let $\lambda \ge \LS (K)$ and $\alpha$ be cardinals. We say that $K$ is \emph{$\alpha$-stable in $\lambda$} if for any $M \in K_\lambda$, $|\Ss^\alpha (M)| \le \lambda$.

  We say that $K$ is \emph{stable} in $\lambda$ if it is $1$-stable in $\lambda$.

  We say that $K$ is \emph{$\alpha$-stable} if it is $\alpha$-stable in $\lambda$ for some $\lambda \ge \LS (K)$. We say that $K$ is \emph{stable} if it is $1$-stable in $\lambda$ for some $\lambda \ge \LS (K)$. We write ``unstable'' instead of ``not stable''. 

  We define similarly stability for $K_{\F}$, e.g.\ $K_{\F}$ is stable if and only if $K$ is stable in $\lambda$ for some $\lambda \in \F$.
\end{defin}

\begin{remark}\label{monot-stab}
  If $\alpha < \beta$, and $K$ is $\beta$-stable in $\lambda$, then $K$ is $\alpha$-stable in $\lambda$.
\end{remark}

The following follows from \cite[Theorem 3.1]{longtypes-toappear-v2}.

\begin{fact}\label{stab-longtypes}
  Let $\lambda \ge \LS (K)$. Let $\alpha$ be a cardinal. Assume $K$ is stable in $\lambda$ and $\lambda^\alpha = \lambda$. Then $K$ is $\alpha$-stable in $\lambda$.
\end{fact}

\subsection{Commutative Diagrams} Since a picture is worth a thousand words, we make extensive use of commutative diagrams to illustrate the proofs.  Most of the notation is standard.  When we write
\[
\xymatrix{
M_0 \ar[r]_{[a]}^f & M_1 \ar@{.>}[r]_{[\bb]}^g & M_2
}
\]
The functions $f$ and $g$, typically written above arrows, are always $K$-embeddings; that is, $f: M_0 \cong f[M_0] \lea M_1$.  Writing no functions means that the $K$-embedding is the identity.  The elements in square brackets $a$ and $\bb$, typically written below arrows, are elements that exist in the target model, but not the source model; that is, $a \in M_1 - f[M_0]$.  Writing no element simply means that there are no elements \emph{that we wish to draw the reader's attention to} in the difference.  In particular, it does \emph{not} mean that the two models are isomorphic.  We sometimes make a distinction between embeddings appearing in the hypothesis of a statement (denoted by solid lines),  and those appearing in the conclusion (denoted by dotted lines).

\section{Good frames}\label{good-frame}

Good frames were first defined in \cite[Chapter II]{shelahaecbook}. The idea is to provide a localized (i.e.\ only for base models of a given size $\lambda$) axiomatization of a forking-like notion for a ``nice enough'' set of 1-types.  These axioms are similar to the properties of first-order forking in a superstable theory. Jarden and Shelah (in \cite{jrsh875}) later gave a slightly more general definition, not assuming the existence of a superlimit model and dropping some of the redundant clauses. We give a slightly more general variation here: following \cite{ss-tame-jsl}, we assume the models come from $K_{\F}$, for $\F$ an interval, instead of just $K_\lambda$. We also assume that the types could be longer than just types of singletons. We first adapt the definition of a pre-$\lambda$-frame from \cite[Definition III.0.2.1]{shelahaecbook}:

\begin{defin}[Pre-frame]
  Let $\alpha$ be an ordinal and let $\F$ be an interval of the form $[\lambda, \theta)$, where $\lambda$ is a cardinal, and $\theta > \lambda$ is either a cardinal or $\infty$.

  A \emph{pre-$(<\alpha, \F)$-frame} is a triple $\s = (K, \nf, \Sbs)$, where:

  \begin{enumerate}
  \item $K$ is an abstract elementary class with $\lambda \ge \LS (K)$, $K_\lambda \neq \emptyset$.
  \item $\Sbs \subseteq \bigcup_{M \in K_{\F}} \Sana{<\alpha} (M)$. For $M \in K_{\F}$ and $\beta$ an ordinal, we write $\Sabs{\beta} (M)$ for $\Sbs \cap \Sana{\beta} (M)$ and similarly for $\Sabs{<\beta} (M)$.
  \item $\nf$ is a relation on quadruples of the form $(M_0, M_1, \ba, N)$, where $M_0 \lea M_1 \lea N$, $\ba \in \fct{<\alpha}{N}$, and $M_0$, $M_1$, $N$ are all in $K_{\F}$. We write $\nf(M_0, M_1, \ba, N)$ or $\nfs{M_0}{\ba}{M_1}{N}$ instead of $(M_0, M_1, a, N) \in \nf$. 
  \item The following properties hold:
    \begin{enumerate}
      \item \underline{Invariance}: If $f: N \cong N'$ and $\nfs{M_0}{\ba}{M_1}{N}$, then $\nfs{f[M_0]}{f(\ba)}{f[M_1]}{N'}$. If $\tp (\ba / M_1; N) \in \Sbs (M_1)$, then $\tp(f (\ba) / f[M_1]; N') \in \Sbs (f[M_1])$.
      \item \underline{Monotonicity}: If $\nfs{M_0}{\ba}{M_1}{N}$, $\ba'$ is a subsequence of $\ba$, $M_0 \lea M_0' \lea M_1' \lea M_1 \lea N' \lea N \lea N''$ with $\ba' \in N'$, and $N'' \in K_{\F}$, then $\nfs{M_0'}{\ba'}{M_1'}{N'}$ and $\nfs{M_0'}{\ba'}{M_1'}{N''}$. If $\tp (\ba / M_1; N) \in \Sbs (M_1)$ and $\ba'$ is a subsequence of $\ba$, then $\tp (\ba' / M_1; N) \in \Sbs (M_1)$.
      \item \underline{Nonforking types are basic}: If $\nfs{M}{\ba}{M}{N}$, then $\tp (\ba / M; N) \in \Sbs (M)$.
    \end{enumerate}
  \end{enumerate}

  A \emph{pre-$(\le \alpha, \F)$-frame} is a pre-$(<(\alpha+1), \F)$-frame. When $\alpha = 1$, we drop it. We write pre-$(<\alpha, \lambda)$-frame instead of pre-$(<\alpha, \{\lambda\})$-frame or pre-$(<\alpha, [\lambda, \lambda^+))$-frame; and pre-$(<\alpha, (\ge \lambda))$-frame instead of pre-$(<\alpha, [\lambda, \infty))$-frame. We sometimes drop the $(<\alpha, \F)$ when it is clear from context.

    For $\s$ a pre-$(<\alpha, \F)$-frame, $\beta \le \alpha$, and $\F' \subseteq \F$ an interval, we let $\s_{\F'}^{<\beta}$ denote the pre-$(<\beta, \F')$-frame defined in the obvious way by restricting the basic types and $\nf$ to models in $K_{\F'}$ and elements of length $<\beta$. If $\F' = \F$ or $\beta = \alpha$, we omit it. For $\lambda' \in \F$, we write $\s_{\lambda'}^{<\beta}$ instead of $\mathfrak{s}_{\{\lambda'\}}^{<\beta}$.
\end{defin}
\begin{remark}
  Note that, following Shelah's original definition, we have defined nonforking (in the sense of frames) only for nonalgebraic types.  However, this restriction is inessential: We could expand the definition of nonforking to algebraic types by saying that an algebraic $p \in S(M)$ does not fork over $M_0$ if and only if $p \rest M_0$ is algebraic.  This change would not affect whether or not a frame satisfies the properties given\footnote{In the statement of the extension property, we would need to require that the nonforking extension of a nonalgebraic type is nonalgebraic.}.
\end{remark}
\begin{remark}
  The reader might wonder about the reasons for having a special class of basic types. Following Shelah \cite[Definition III.9.2]{shelahaecbook}, let us call a pre-frame \emph{type-full} if the basic types are all the nonalgebraic types. It can be shown \cite[III.9.6]{shelahaecbook} that any weakly successful good frame can be extended to a type-full one. Furthermore, there are no known examples of a good $\lambda$-frame which which cannot be extended to a type-full one. However Shelah's initial construction \cite[II.3]{shelahaecbook} builds a non type-full good frame and it is not clear that it can be extended to a type-full one until after Shelah shows that the frame is weakly successful. Thus it can be easier to build a good frame than to build a type-full one, and most results about frame already hold in the non-type-full context. In this paper, we will set the basic types to be the independent sequences, hence getting another natural example of a non type-full good frame.
\end{remark}

\begin{notation}
If $\s = (K, \nf, \Sbs)$ is a pre-$(<\alpha, \F)$-frame, then $\alpha_\s := \alpha$, $\F_s := \F$, $K_\s := K$, $\nf_\s := \nf$, and $\Sbs_\s := \Sbs$. If $\F = [\lambda, \theta)$, then let $\lambda_\s := \lambda$, $\theta_\s := \theta$.
\end{notation}

By the invariance and monotonicity properties, $\nf$ is really a relation on types. This justifies the next definition.

\begin{defin}
  If $\s = (K, \nf, \Sbs)$ is a pre-$(<\alpha, \F)$-frame, $p \in \Ss^{<\alpha} (M_1)$ is a type, we say $p$ \emph{does not fork over $M_0$} if $\nfs{M_0}{\ba}{M_1}{N}$ for some (equivalently any) $\ba$ and $N$ such that $p = \tp (\ba / M_1; N)$.  If $\s$ is not clear from context, we add ``with respect to $\s$''.
\end{defin}

\begin{remark}
  We could have started from $(K, \nf)$ and defined the basic types as those that do not fork over their own domain. Since we are sometimes interested in studying frames that only satisfy existence over a certain class of models (like the saturated models), we will not adopt this approach.
\end{remark}
\begin{remark}
  We could also have specified only $K_\F$ or even only $K_\lambda$ instead of the full AEC $K$. This is completely equivalent since, by \cite[Section II.2]{shelahaecbook}, $K_\lambda$ fully determines $K$.
\end{remark}

\begin{defin}[Good frame]\label{good-frame-def}
  Let $\alpha$, $\F$ be as above.

  A \emph{good $(<\alpha, \F)$-frame} is a pre-$(<\alpha, \F)$-frame $(K, \nf, \Sbs)$ satisfying in addition:

  \begin{enumerate}
  \item $K_{\F}$ has amalgamation, joint embedding, and no maximal models.
    \item \underline{bs-Stability}: $|\Sabs{1} (M)| \le \|M\|$ for all $M \in K_{\F}$.
    \item \underline{Density of basic types}: If $M \lta N$ are in $K_{\F}$, then there is $a \in N$ such that $\tp (a / M; N) \in \Sbs (M)$.
  \item \underline{Existence of nonforking extension}: If $p \in \Sbs (M)$, $N \gea M$ is in $K_{\F}$, and $\beta < \alpha$ is such that $\ell(p) \leq \beta$, then there is some $q \in \Sabs{\beta}(N)$ that does not fork over $M$ and extends $p$, i.e. $q^\beta \rest M = p$.
  \item \underline{Uniqueness}: If $p, q \in \Ss^{<\alpha} (N)$ do not fork over $M$ and $p \upharpoonright M = q \upharpoonright M$, then $p = q$.
  \item \underline{Symmetry}: If $\nfs{M_0}{\ba_1}{M_2}{N}$, $\ba_2 \in \fct{<\alpha}{M_2}$, and $\tp (\ba_2 / M_0; N) \in \Sbs (M_0)$, then there is $M_1$ containing $\ba_1$ and $N' \gea N$ such that $\nfs{M_0}{\ba_2}{M_1}{N'}$.
  \item \underline{Local character}: If $\delta$ is a regular cardinal, $\seq{M_i \in K_{\F} : i \le \delta}$ is increasing continuous, and $p \in \Sbs (M_\delta)$ is such that $\ell (p) < \delta$, then there exists $i < \delta$ such that $p$ does not fork over $M_i$.
  \item \underline{Continuity}: If $\delta$ is a limit ordinal, $\seq{M_i \in K_{\F}: i \leq \delta}$ and $\seq{\alpha_i < \alpha : i \leq \delta}$ are increasing and continuous, and $p_i \in \Sabs{\alpha_i} (M_i)$ for $i < \delta$ are such that $j < i < \delta$ implies $p_j = p_i^{\alpha_j} \rest M_j$, then there is some $p \in \Sabs{\alpha_\delta} (M_\delta)$ such that for all $i < \delta$, $p_i = p^{\alpha_i} \rest M_i$.  Moreover, if each $p_i$ does not fork over $M_0$, then neither does $p$.
  \item \underline{Transitivity}: If $M_0 \lea M_1 \lea M_2$, $p \in S (M_2)$ does not fork over $M_1$ and $p \upharpoonright M_1$ does not fork over $M_0$, then $p$ does not fork over $M_0$.
  \end{enumerate}

  We will sometimes refer to ``existence of nonforking extension'' as simply ``existence''.

  For $\mathbb{L}$ a list of properties\footnote{This notation was already used in \cite[Definition 2.21]{ss-tame-jsl}.}, a $\goodms{\mathbb{L}}$ $(<\alpha, \F)$-frame is a pre-$(<\alpha, \F)$-frame that satisfies all the properties of good frames except possibly the ones in $\mathbb{L}$. In this paper, $\mathbb{L}$ will only contain symmetry and/or bs-stability. We abbreviate symmetry by $S$, bs-stability by $St$, and write $\goodm$ for $\goodms{(S, St)}$.

  We say that $K$ has a good $(<\alpha, \F)$-frame if there is a good $(<\alpha, \F)$-frame where $K$ is the underlying AEC (and similarly for $\goodm$).
\end{defin}
\begin{remark}\label{trans-prop}
  Transitivity follows directly from existence and uniqueness by \cite[Claim II.2.18]{shelahaecbook}.
\end{remark}

\begin{remark}
  The obvious monotonicity properties hold: If $\s$ is a good $(<\alpha, \F)$-frame, $\beta \le \alpha$,  and $\F'$ is a subinterval of $\F$, then $\s_{\F'}^{<\beta}$ is a good $(<\beta, \F')$ frame (and similarly for $\goodm$).
\end{remark}
\begin{remark}\label{good-frame-existence}
  If $T$ is a superstable first-order theory, then forking induces a good $(\ge |T|)$-frame on the class of models of $T$ ordered by elementary submodel. In the non-elementary context, Shelah showed in \cite[Theorem II.3.7]{shelahaecbook} how to build a good frame from local categoricity hypotheses and GCH-like assumptions, while the second author \cite{ss-tame-jsl} showed how to build a good frame in ZFC from categoricity, tameness, and a monster model. Note that a family of examples due to Hart and Shelah \cite{hs-example} demonstrates that, in the absence of tameness, an AEC could have a good $\lambda$-frame but no good $(\ge \lambda)$-frame (see \cite[Section 10]{ext-frame-jml} for a detailed writeup).
\end{remark}

Note that for types of finite length, local character implies that nonforking is witnessed by a model of small size:

\begin{prop}\label{kappabar}
  Let $\alpha \le \omega$. Assume $\s = (K, \nf, \Sbs)$ is a pre-$(<\alpha, \F)$-frame satisfying local character and transitivity. If $M \in K_{\F}$ and $p \in \Sbs (M)$, then there exists $M' \in K_\lambda$ such that $p$ does not fork over $M'$.
\end{prop}
\begin{proof}
  Same proof as \cite[Proposition 2.23]{ss-tame-jsl} (there $\alpha = 1$ but this does not change the proof).
\end{proof}




We conclude this section with an easy variation on the existence property that will be used later.

\begin{lem}\label{ext-vars}
  Assume $\s = (K, \nf, \Sbs)$ is a pre-$(<\alpha, \F)$-frame with amalgamation, existence, and transitivity. Suppose $M \lea M_0 \lea M_1$ are in $K_{\F}$ and $f: M_0 \rightarrow M_2$ is given with $M_2 \in K_{\F}$. Assume also that we have $\ba \in M_1$ such that $\nfs{M}{\ba}{M_0}{M_1}$.


  There is $N \gea M_2$ and $g: M_1 \to N$ extending $f$ such that $\nfs{g[M]}{g (\ba)}{M_2}{N}$. A diagram is below. 

    \[
    \xymatrix{ & M_1 \ar@{.>}[r]_g & N\\
      & M_0 \ar[u]^{[\ba]} \ar[r]_f & M_2 \ar@{.>}[u]\\
      M \ar[ur] & & 
    }
    \]

\end{lem}
\begin{proof}


Extend $f$ to an $L(K)$-isomorphism $\bigf$ with range $M_2$.  By existence, there is some $q \in \Sbs(\bigf^{-1}[M_2])$ that extends $\tp(\ba/M_0; M_1)$ and does not fork over $M_0$.  Realize $q$ as $\tp(\bb/\bigf^{-1}[M_2]; N^+)$.  Since $\tp(\ba/M_0; M_1) = \tp(\bb/M_0; N^+)$, there is $N^{++} \gea N^+$ and $h:M_1 \xrightarrow[M_0]{} N^{++}$ such that $h(\ba) = \bb$.  Then, since $N^{++}$ extends $\bigf^{-1}[M_2]$, we can find an $L(K)$-isomorphism $\bigf^+$ that extends $\bigf$ such that $N^{++}$ is the domain of $\bigf^+$.  Set $N:= \bigf^+[N^{++}]$ and $g := \bigf^+ \circ h$.  Some nonforking calculus shows that this works. 
\end{proof}

\section{Independent sequences form a $\goodm$ frame}\label{indep-seq-good}

In this section, we show how to make a $\goodms{S}$ $\F$-frame longer (i.e.\ extend the nonforking relation to longer sequences).  This is done by using independent sequences, introduced by Shelah \cite[Definition III.5.2]{shelahaecbook} and also studied by Jarden and Sitton \cite{jasi}, to define basic types and nonforking. Preservation of the symmetry property will be studied in Section \ref{sym-long-frames}, and in Section \ref{going-up} we will review how to make the frame larger (i.e.\ extend the nonforking relation to larger \emph{models}). 

Note that Shelah already claims many of the results of this section (for finite tuples) in \cite[Exercise III.9.4.1]{shelahaecbook} but the proofs have never appeared anywhere.

\begin{defin}[Independent sequence] \label{indep-seq-def}
Let $\alpha$ be an ordinal and let $\s$ be a pre-$\F$-frame. 

  \begin{enumerate}
  \item $\seq{a_i:i < \alpha}$, $\seq{M_i : i \leq \alpha}$ is said to be \emph{independent in $(M, M', N)$} when:

    \begin{enumerate}
    \item $(M_i)_{i \leq \alpha}$ is increasing continuous in $K_{\F}$.
    \item $M \lea M' \lea M_0$ and $M, M' \in K_{\F}$.
    \item $M_\alpha \lea N$ is in $K_{\F}$.
    \item For every $i < \alpha$ , $\nfs{M}{a_i}{M_i}{M_{i+1}}$.
    \end{enumerate}

    $\seq{a_i : i < \alpha}$, $\seq{M_i: i \leq \alpha}$ is said to be \emph{independent over $M$} when it is independent in $(M, M_0, M_\alpha)$.
  \item $\ba := \seq{a_i : i < \alpha}$ is said to be \emph{independent in $(M, M', N)$} when for some $\seq{M_i : i \le \alpha}$ we have that $\seq{a_i : i < \alpha}$, $\seq{M_i : i \leq \alpha}$ is independent in $(M, M', N)$.
  \item We say that $\seq{a_i : i < \alpha}$, $\seq{M_i : i \leq \alpha}$ is \emph{independent from $M'$ over $M$ in $N$} if it is independent in $(M, M', N)$. We similarly define \emph{$\ba$ is independent from $M'$ over $M$ in $N$}. When $N$ is clear from context, we drop it.
  \end{enumerate}
\end{defin}
\begin{remark}
  If $\alpha = 1$, then $\ba := \seq{a_0}$ is independent from $M'$ over $M$ in $N$ if and only if $\tp (a_0 / M'; N)$ does not fork over $M$. 
\end{remark}

This motivates the next definition:

\begin{defin}
Let $\s := (K, \nf, \Sbs)$ be a pre-$\F$-frame, where $\F = [\lambda, \theta)$. Let $\alpha \le \theta$. Define $\s^{<\alpha} := (K, \nf^{<\alpha}, \Sabs{<\alpha})$ as follows:

  \begin{itemize}
  \item For $M_0 \lea M_1 \lea N$ in $K_{\F}$ and $\ba := (a_i)_{i < \beta}$ in $N$ with $\beta < \alpha$, $\nf^{<\alpha} (M_0, M_1, \ba,  N)$ if and only if $\ba$ is independent from $M_1$ over $M_0$ in $N$.
  \item For $M \in K_{\F}$ and $p \in \Ss^{<\alpha} (M)$, $p \in \Sabs{<\alpha} (M)$ if and only if there exists $N \gea M$ and $\ba \in N$ such that $p = \tp (\ba / M; N)$ and $\nf^{<\alpha} (M, M, \ba, N)$.
  \end{itemize}
\end{defin}

\begin{lem}[Invariance]\label{type-indep}
Let $\s := (K, \nf, \Sbs)$ be a pre-$\F$-frame, where $\F = [\lambda, \theta)$. Let $\alpha \le \theta$. Assume $K_{\F}$ has amalgamation. Given $\ba = \seq{a_i : i < \alpha}$ independent from $M_0$ over $M$ in $M_1$ and $M_2 \gea M_0$ containing $\bb$ such that $\tp(\ba/M_0; M_1) = \tp(\bb/M_0; M_2)$, we have that $\bb$ is independent from $M_0$ over $M$ in $M_2$.
\end{lem}
\begin{proof}
  Straightforward.

\end{proof}

\begin{remark}
When dealing with types rather than sequences, the $N^+$ in the definition can be avoided. That is, given $p \in \Sabs{\beta} (N)$ that does not fork over $M$, there is some $\seq{a_i : i < \beta}$, $\seq{N^i: i \leq \beta}$ such that $p = \tp(\seq{a_i : i < \beta}/N; N^\beta)$ that witnesses that $\seq{a_i : i < \beta}$ is independent from $N$ over $M$ in $N^\beta$.
\end{remark}

\begin{lem}\label{pre-frame-transfer}
Let $\s := (K, \nf, \Sbs)$ be a pre-$\F$-frame, where $\F = [\lambda, \theta)$. Let $\alpha \le \theta$. If $K_{\F}$ has amalgamation, then $\s^{<\alpha}$ is a pre-$(<\alpha, \F)$-frame.
\end{lem}
\begin{proof}
  Invariance is Lemma \ref{type-indep}. For monotonicity, one can also use invariance to see that if $\ba$ is independent from $M_1$ over $M_0$ in $N$ and $N' \gea N$, then $\ba$ is independent from $M_1$ over $M_0$ in $N'$. The rest is straightforward.
\end{proof}

The next result shows that local character and existence are preserved when elongating a frame:

\begin{thm}\label{lc-ext-transfer} 
  Assume $\s := (K, \nf, \Sbs)$ is a $\goodm$ $\F$-frame, where $\F = [\lambda, \theta)$. Then:
  
  \begin{enumerate}
  \item\label{lc-transfer} $\s^{<\theta}$ has local character. Moreover, if $p \in \Sabs{\alpha} (N)$ with $\alpha < \theta$, then there exists $M \lea N$ in $K_{\le \lambda + |\alpha|}$ such that $p$ does not fork over $M$.
  \item $\s^{<\theta}$ has existence.
  \end{enumerate}
\end{thm}
\begin{proof} \
\begin{enumerate}
\item Assume $p \in  \Sabs{\alpha} (N)$ and $N = \bigcup_{i < \delta}N_i$ with $\alpha < \delta < \theta$, $\delta$ a regular cardinal. Then, there is some $\ba = \seq{a_i : i < \alpha}$ and increasing, continuous $\seq{N^i : i \leq \alpha}$ such that $\alpha < \delta$, $p = \tp(\ba/N; N^\alpha)$, and, for all $i < \alpha$, $\nfs{N}{a_i}{N^i}{N^{i+1}}$.  By Monotonicity for $\s$, $\tp(a_i/N; N^{i+1}) \in \Sbs (N)$. By Local Character for $\s$, for all $i < \alpha$ there is some $j_i < \delta$ such that $\nfs{N_{j_i}}{a_i}{N}{N^{i+1}}$.  By Transitivity for $\s$, $\nfs{N_{j_i}}{a_i}{N^i}{N^{i+1}}$.  Set $j_* := \sup_{i < \alpha} j_i$; since $\delta = \cf{\delta} > \alpha$, we have that $j_* < \alpha$.  By Monotonicity for $\s$, $\nfs{N_{j_*}}{a_i}{N^i}{N^{i+1}}$ for all $i < \alpha$.  This is exactly what we need to conclude that $\ba$ is independent from $N$ over $N_{j_*}$ in $N^\alpha$.  Thus $p = \tp(\ba/N; N^\alpha)$ does not fork over $N_{j_*}$. 

  The moreover part is proved similarly: By Proposition \ref{kappabar}, for each $i < \alpha$ there exists $M^i \lea N$ in $K_\lambda$ such that $\nfs{M^i}{N}{a_i}{N^\alpha}$. By Transitivity, $\nfs{M^i}{N^i}{a_i}{N^\alpha}$. Now by the Löwenheim-Skolem axiom, there exists $M \lea N$ in $K_{\le \lambda + |\alpha|}$ such that $\bigcup_{i < \alpha} {M^i} \lea M$. By Monotonicity, $\nfs{M}{N^i}{a_i}{N^\alpha}$, so $\ba$ is independent from $N$ over $M$ in $N^\alpha$, so $p$ does not fork over $M$, as needed.
	
\item We prove two extension results separately: extending the domain and extending the length.  Combining these two results shows that $\s^{<\theta}$ has existence.
	
  For extending the domain, let $p \in \Sabs{<\theta} (M)$ and $N \gea M$.  By definition of this frame, there is some $\ba=\seq{a_i : i < \beta}$ and increasing, continuous $\seq{M^i :i \leq \beta}$ such that $\nfs{M}{a_i}{M^i}{M^{i+1}}$ for all $i < \beta$.  We wish to construct increasing and continuous $\seq{N^i : i \leq \beta}$ and $\seq{f_i : M^i \to N^i : i \leq \beta}$ such that 
	\begin{enumerate}
		\item $f_0 \rest M = \id$; and
		\item $\nfs{M}{f_i (a_i)}{N^i}{N^{i+1}}$.
	\end{enumerate}
	This is done by induction by taking unions at limits and by using Lemma \ref{ext-vars} at all successor steps.  Since $\beta < \theta$, $N^i$ is in $K_{\F}$ at all steps and the induction can continue.  Then $\tp(\ba/M; M^\beta) = \tp(f(\ba)/M; N^\beta)$ as witnessed by $f$ and $f(\ba)$ is independent in $(M, N, N^\beta)$.  Thus, $q = \tp(f(\ba)/N, N^\beta)$ is as desired.
	\[
\xymatrix{
N \ar@{.>}[r] & N^i \ar@{.>}[r] & N^{i+1} \ar@{.>}[r] & N^\beta\\
M \ar[u] \ar[r] & M^i \ar[r]  \ar@{.>}[u]^{f_i} & M^{i+1} \ar[r] \ar@{.>}[u]^{f_{i+1}} & M^\beta \ar@{.>}[u]^{f_\beta}
}
\]
	
	To extend the length, suppose that $\beta < \alpha < \theta$ and $p \in \Sabs{\beta} (N)$ does not fork over $M$.  This means that there is $\seq{a_i : i< \beta}$, $\seq{N^i : i \leq \beta}$ independent in $(M, N, N^\beta)$ such that $p = \tp(\seq{a_i : i < \beta}/N; N^\beta)$.  We will extend this sequence to be of length $\alpha$ by induction. At limit steps, simply taking the union of the extensions works.  If we have $\beta \leq \gamma < \alpha$ and have already extended to $\gamma$ (i.e.,  $\seq{a_i : i< \gamma}$, $\seq{N^i : i \leq \gamma}$ is defined), then let $r \in \Sbs (M)$ be arbitrary (use no maximal models and density of basic types).  Let $r^+ \in \Sbs (N^\gamma)$ be its nonforking extension.  Thus, there is $a_\gamma \in N^{\gamma+1}$ that realizes $r^+$ such that $\nfs{M}{a_\gamma}{N^{\gamma}}{N^{\gamma+1}}$.  Then $\seq{a_i : i < \gamma+1}$, $\seq{N^i:i \leq \gamma + 1}$ is independent from $N$ over $M$ in $N^{\gamma + 1}$, as desired.
\end{enumerate}
\end{proof}

The next technical lemma is key in showing that uniqueness and continuity are preserved when making a frame longer.  This allows us to put together two independent sequences into one.

\begin{lem}[Amalgamation of independent sequences]\label{amal-indep-seq}
Let $\s$ be a $\goodm$ $\F$-frame, and $\beta < \theta_\s$.  Suppose that $p, q \in \Sabs{\beta}(N)$ do not fork over $M$, that $p \rest M = q \rest M$, and that there are witnessing sequences $\ba_\ell = \seq{a_\ell^i : i < \beta}$, $\seq{N_\ell^i : i \leq \beta}$ independent from $N$ over $M$ in $N_\ell^{\beta}$ for $\ell = 0,1$  with $\ba_0 \vDash p$ and $\ba_1 \vDash q$.  Then, there are coherent, continuous, increasing $(N_i, f_{j, i})_{j<i \leq \beta}$ and $g_\ell^i : N_\ell^i \to N_i$ such that, for all $j < i < \beta$,
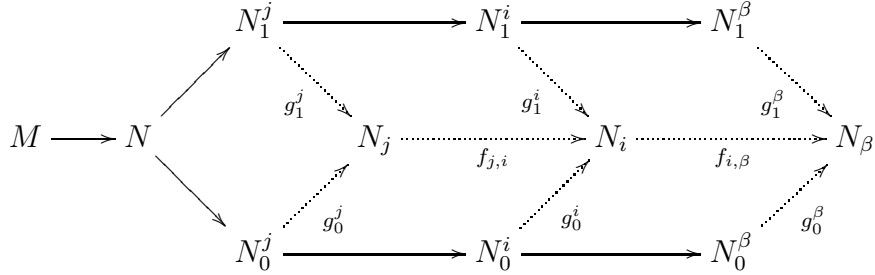
\[
\xymatrix{
& & N^j_1 \ar@{.>}[dr]_{g_1^j} \ar[rr] & & N_1^i \ar@{.>}[dr]_{g_1^i} \ar[rr] & & N_1^\beta \ar@{.>}[dr]_{g_1^\beta} &
\\
M \ar[r] & N \ar[ur] \ar[dr] & & N_j \ar@{.>}[rr]_{f_{j, i}} & & N_i \ar@{.>}[rr]_{f_{i, \beta}} & & N_\beta\\
& & N^j_0 \ar@{.>}[ur]_{g_0^j} \ar[rr] & & N_0^i \ar@{.>}[ur]_{g_0^i} \ar[rr] & & N_0^\beta \ar@{.>}[ur]_{g_0^\beta} &
}
\]
commutes, $g_0^{i+1}(a_0^i) = g_1^{i+1}(a_1^i)$, and\footnote{Note that $g_0^i[M] = g_1^i[M]$ by commutativity of the diagram.} $\nfs{g_0^{i+1}[M]}{g_0^{i+1}(a_0^i)}{f_{i, i+1}[N_i]}{N_{i+1}}$.
\end{lem}
\begin{proof}
We will build:
	\begin{enumerate}
		\item models $\{N_i, M_\ell^i : i \leq \beta, \ell = 0, 1\}$;
		\item embeddings $\{h_\ell^i : N_\ell^i \to M_\ell^i, r_\ell^i:M_\ell^i \to N_i : i \leq \beta, \ell = 0,1\}$; and 
		\item coherent embeddings $\{f_{j, i}:N_j \to N_i, \hatr_\ell^{j, i}: M_\ell^j \to M_\ell^i : i \leq \beta, \ell=0, 1\}$
	\end{enumerate}
	such that, for $i < \beta$:
	\begin{enumerate}
	
		\item  \[
	\xymatrix{
	M_0^{i+1} \ar[r]_{r_0^{i+1}} & N_{i+1}\\
	N_i \ar[u] \ar[r] & M_1^{i+1} \ar[u]_{r_1^{i+1}}
	}
	\]
	commutes;
	
		\item\label{second-comm} \[
		\xymatrix{
		N_\ell^{i+1} \ar[rr]_{h_\ell^{i+1}}& & M_\ell^{i+1}\\
		N_\ell^i \ar[u] \ar[r]_{h_\ell^i}& M_\ell^i \ar[r]_{r_\ell^i} & N_i \ar[u]
		}
		\]
		commutes;
		
		\item $M_\ell^0 = N_0$, $r_\ell^0 = \id_{N_0}$ for $\ell=0, 1$, and 
		\[
		\xymatrix{
		N_0^0 \ar[r]_{h_0^0} & N_0\\
		N \ar[u] \ar[r] & N_1^0 \ar[u]_{h_1^0}
		}
		\]
		commutes;
		
		\item \label{4above}$\nfs{h_\ell^{i+1}[N_\ell^i]}{h_\ell^{i+1}(a_\ell^i)}{N_i}{M_\ell^{i+1}}$;
		
		\item \label{5above}$r_0^{i+1} \circ h_0^{i+1} (a_0^i) = r_1^{i+1} \circ h_1^{i+1}(a_1^i)$; and 
		
		\item \label{6above}$(N_i, f_{j, i})_{j < i \leq \beta}$ and $(M_\ell^i, \hatr_\ell^{j, i})_{j < i \leq \beta}$ are continuous, coherent systems generated by $\hatr_\ell^{i, i+1} = r_\ell^i$ and $f_{i, i+1} = r_0^i \upharpoonright N_i = r_1^i \upharpoonright N_i$.
		
	\end{enumerate}

        Once these objects have been constructed we will have the following commutative diagram  for $j < i \le \beta$:
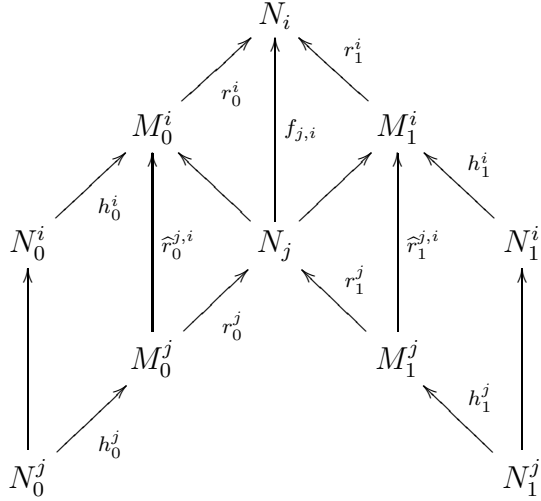
        \[
        \xymatrix{
	& & N_i & & \\
	& M_0^i \ar[ur]_{r_0^i} & & M_1^i \ar[ul]_{r_1^i} & \\
	N_0^i \ar[ur]_{h_0^i} & & N_{j} \ar[uu]_{f_{j, i}} \ar[ur] \ar[ul] & & N_1^i \ar[ul]_{h_1^i}\\
	& M_0^j \ar[uu]_{\hatr_0^{j, i}} \ar[ur]_{r_0^j} & & M_1^j \ar[uu]_{\hatr_1^{j,i}} \ar[ul]_{r_1^j} & \\
	N_0^j \ar[uu] \ar[ur]_{h_0^j} & & & & N_1^j \ar[ul]_{h_1^j} \ar[uu]
	}
	\]

        We can then take $g_\ell^i := r_\ell^i \circ h_\ell^i$.  This gives the desired diagram by removing the $M^i_\ell$'s.  The function equality is given by (\ref{5above}) and the nonforking is given by applying $f_{i, i+1}$ to (\ref{4above}).
	
The construction proceeds by induction.  At stage $i$, we will construct $h_\ell^i, r_\ell^i, M_\ell^i$, and $N_i$ for $\ell = 0, 1$.  Also, at each stage, we implicitly extend the coherent system by the rule given in (\ref{6above}) above (at successor steps) or by taking direct limits (at limit steps).\\
	\underline{$i=0$:} Amalgamate $N_0^0, N_1^0$ over $N$ to get $N_0$. Also set $M_\ell^0 := N_0$ and $r_\ell^0 := \id_{N_0}$ for $\ell=0, 1$. \\
	\underline{$i$ limit:}  Take direct limits and use continuity to see everything is preserved. \\
	\underline{$i=j+1$:} Use Lemma \ref{ext-vars} --replace $(M, M_0, M_1, \ba, f, M_2)$ there with $(M, N_\ell^j, N_\ell^{j + 1}, a_\ell^j, r_\ell^j \circ h_\ell^j, N_j)$ here--to get $(h_\ell^{j+1}, M_\ell^{j+1})$ here, written as $(g, N)$ there
; this gives (\ref{4above}):
	$$\nfs{h_\ell^{j+1}[M]}{h_\ell^{j+1}(a_\ell^j)}{N_j}{M_\ell^{j+1}}$$
        
	By the commutative diagrams, $h_0^{j+1} \rest M = h_1^{j+1} \rest M$, so, since $a_0^j$ and $a_1^j$ have the same type over $M$, we have that:
	$$\tp(h_0^{j+1}(a_0^j)/h_0^{j+1}[M]; M_0^{j+1}) = \tp(h_1^{j+1}(a_1^j)/h_1^{j+1}[M]; M_1^{j+1})$$
	By Uniqueness for $\s$, these imply that:
	$$\tp(h_0^{j+1}(a_0^j)/N_j; M_0^{j+1}) = \tp(h_1^{j+1}(a_1^j)/N_j; M_1^{j+1})$$
	We can witness this with $r_\ell^{j+1}:M_\ell^{j+1} \to N_{j+1}$ for $\ell=0,1$; that is, $r_0^{j+1} \rest N_j = r_1^{j+1} \rest N_j$ and $r_0^{j+1} \circ h_0^{j+1}(a_0^j) = r_1^{j+1} \circ h_1^{j+1}(a_1^j)$.
\end{proof}

\begin{cor} \label{extend-indep-seq}
Let $\s = (K, \nf, \Sbs)$ be a $\goodm$ $\F$-frame. Suppose $M_0 \lea M \lea N$ are in $K_{\F}$ and $\alpha \leq \beta < \theta_{\s}$ are such that there are $p \in \Sabs{\alpha}(M)$ and $q \in \Sabs{\beta}(N)$ such that $q^\alpha \rest M = p$ and $p, q$ do not fork over $M_0$. If $\ba_p = \seq{a_p^i : i < \alpha}$, $\seq{N_p^i : i \leq \alpha}$ is independent from $M$ over $M_0$ in $N_p^\alpha$ such that $\ba_p \vDash p$ and $\ba_q = \seq{a_q^i : i < \beta}$, $\seq{N_q^i : i \leq \beta}$ is independent from $N$ over $M_0$ in $N_q^\beta$ such that $\ba_q \vDash q$, then there is $\seq{M_q^i : i \leq \beta}$ and $h_i: N_p^i \to M_q^i$ for $i \leq \alpha$ such that:
\begin{enumerate}
	\item $\ba_q$, $\seq{M_q^i : i \leq \beta}$ is independent from $N$ over $M_0$ in $M_q^\beta$;
	\item $N_q^i \lea M_q^i$ for all $i \leq \beta$; and 
	\item $h_{i+1}(a_p^i) = a_q^i$ and $\id_M \subseteq h_i \subseteq h_{i+1}$.
\end{enumerate}
\end{cor}
\begin{proof}
  First, extend the $p$-sequence to $\seq{a_p^i : i < \beta}$, $\seq{N_p^i : i \leq \beta}$ independent from $M$ over $M_0$ in $N_p^\beta$ (use that $\s^{<\theta_{\s}}$ has existence). We can then amalgamate these sequences over $M$ using Lemma \ref{amal-indep-seq}: there is $(N_i, f_{j, i})_{j<i \leq \beta}$ and $g_x^i:N_x^i \to N_i$ for $x = p, q$ and $i \leq \beta$ as above.  Since we have $g_q^\beta: N_q^\beta \cong g_q^\beta[N_q^\beta] \lea N_\beta$, we can extend $g_q^\beta$ to an $L(K)$-isomorphism $h$ with $N_\beta$ in its range.  Set $M_q^i := h^{-1} [N_i]$ for $i \leq \beta$.  Note that $h_i:= h^{-1} \circ g_q^i :N_q^i \to M_q^i$ is the identity.
\end{proof}

\begin{cor}\label{uniq-cont}
  Assume $\s := (K, \nf, \Sbs)$ is a $\goodm$ $\F$-frame, where $\F = [\lambda, \theta)$. Then:  

\begin{enumerate}
\item $\s^{<\theta}$ has uniqueness.
\item $\s^{<\theta}$ has continuity.
\end{enumerate}
\end{cor}
\begin{proof} \
\begin{enumerate}
\item This follows directly from Lemma \ref{amal-indep-seq}.
\item We prove the moreover clause in the definition of continuity.  For the main clause, the $M_0$'s appearing in this proof can be replaced by $M_i$ or $M_\delta$ as appropriate.

For all $i < \delta$, there is some $\ba_i = \seq{a_i^k : k < \alpha_i}$, $\seq{N_i^k : k \leq \alpha_i}$ independent from $M_i$ over $M_0$ in $N_i^{\alpha_i}$ such that $p_i = \tp(\ba_i/M_i; N_i^{\alpha_i})$.  

We will construct $\seq{M_i^k : i < \delta, k \leq \alpha_i}$ and $\{f_{j, i}^k:M_j^k \to M_i^k : k \leq \alpha_j, j < i < \alpha_\delta\}$ such that
\begin{enumerate}

	\item $N_i^k \lea M_i^k$ and $\ba_i, \seq{M_i^k : k < \alpha_i}$ is independent from $M_i$ over $M_0$ in $M_i^{\alpha_i}$;
	\item for each $k \leq \alpha_j$, $(M_i^k, f_{l, i}^k)_{j \leq l \leq i < \alpha_\delta}$ is a coherent, direct system such that
	\[
	\xymatrix{
	M_{i_2} \ar[r] & M_{i_2}^{k_0} \ar[r] & M_{i_2}^{k_1} \\
	M_{i_1} \ar[r] \ar[u] & M_{i_1}^{k_0} \ar[r] \ar[u]^{f_{i_1, i_2}^{k_0}} & M_{i_1}^{k_1} \ar[u]^{f_{i_1, i_2}^{k_1}}\\
	M_{i_0} \ar[r] \ar[u] & M_{i_0}^{k_0} \ar[r] \ar[u]^{f_{i_0, i_1}^{k_0}} & M_{i_0}^{k_1} \ar[u]^{f_{i_0, i_1}^{k_1}} \ar@/_1pc/[uu]_{f_{i_0, i_2}^{k_1}}
	}
	\]
	commutes; and
	\item $f_{j, i}^k(a_j^k) =a_i^k$. 

\end{enumerate}

This is possible: just apply Corollary \ref{extend-indep-seq} at successors and take direct limits at limits.


This is enough. For each $k < \alpha_\delta$, set $(M_\delta^k, f_{i, \delta}^k)_{i<\delta, k \leq \alpha_i} = \varinjlim (M_i^k, f_{j, i}^k)$.  Then $\seq{M_\delta^k : k < \alpha_\delta}$ is increasing and continuous because each $\seq{M_i^k : k < \alpha_i}$ is.  Set $M_\delta^{\alpha_\delta} := \cup_{k < \alpha_\delta} M_\delta^k$.  For $k < \alpha_i, \alpha_j$, we have that $f_{i, \delta}^{k+1}(a_i^k) = f_{j, \delta}^{k+1}(a_j^k)$.  Thus, there is no confusion in setting $a_\delta^k = f_{i, \delta}^{k+1}(a_i^k)$ for some/any $k < \alpha_i$.  Set $p = \tp(\ba_\delta/M_\delta, M_\delta^{\alpha_\delta})$.

Note that $M_\delta \lea M_\delta^0$; indeed $f_{i, \delta}^k \rest M_i$ is the identity for all $k \leq \alpha_i$.  Thus, we have that 
$$p_i = \tp(\ba_i/M_i; M_i^{\alpha_i}) = \tp(\seq{a_\delta^k:k < \alpha_i}/ M_i; M_\delta^{\alpha_\delta}) = p^{\alpha_i} \rest M_i$$

{\bf Claim:} For all $k < \alpha_\delta$, $\nfs{M_0}{a_\delta^k}{M_\delta^k}{M_\delta^{k+1}}$.

{\bf Proof of Claim:} Given $i < \delta$ and $k < \alpha_i$, we have by construction that $\nfs{M_0}{a_i^k}{M_i^k}{M_i^{k+1}}$.  Applying $f_{i, \delta}^k$ to this, we get $\nfs{M_0}{a_\delta^k}{f_{i, \delta}^k(M_i^k)}{f_{i, \delta}^{k+1}(M_i^{k+1})}$.  By construction, 
$$M_\delta^k = \bigcup_{i < \delta} f_{i, \delta}^k(M_i^k) \text{ and } M_\delta^{k+1} = \bigcup_{i < \delta} f_{i, \delta}^{k+1}(M_i^{k+1})$$
Thus, by Continuity for $\s$, we have, for all $i < \delta$, $\nfs{M_0}{a_\delta^k}{M_\delta^k}{M_\delta^{k+1}}$.

Thus, $\ba_\delta, \seq{M_\delta^k: k \leq \alpha_\delta}$ is independent from $M_\delta$ over $M_0$ in $M^{\alpha_\delta}_\delta$.  So $p \in \Sabs{\alpha_\delta} (M_\delta)$ and extends each $p_i$ as desired.

\end{enumerate}
\end{proof}
\begin{remark}\label{jarden-serial-rmk}
  Note that a special case (when $\F = [\lambda, \lambda^+]$) of the continuity property above is Jarden's $\lambda^+$-continuity of serial independence (see \cite[Definition 5.3]{jarden-tameness-apal}). This allows Jarden's proof that symmetry transfers up (\cite[Theorem 5.4]{jarden-tameness-apal}) to go through without any extra hypotheses.  Another corollary of continuity is what Jarden and Sitton call the \emph{finite continuity property} (see \cite[Definition 8.2]{jasi}).  This is discussed in detail in Section \ref{continuity-subsec}.
\end{remark}

Putting everything together, we obtain that all the property of a $\goodm$ frame transfer to the elongation; recall that $\goodm$ frames are good frames except they might fail stability and/or symmetry. We will later see that symmetry transfers to finite sequences and give conditions under which it transfers to all sequences. 

\begin{cor}\label{tuple-transfer}
  Assume $\s$ is a $\goodm$ $\F$-frame. Then $\s^{<\theta_\s}$ is a $\goodm$ $(<\theta_\s, \F)$-frame.
\end{cor}
\begin{proof}
  Set $\theta := \theta_\s$.  $\s^{<\theta_\s}$ is a pre-$(<\theta, \F)$-frame by Lemma \ref{pre-frame-transfer}. Amalgamation, joint-embedding, no maximal models, and density of basic types hold since they hold in $\s$. Existence and local character hold by Theorem \ref{lc-ext-transfer}, uniqueness and continuity hold by Corollary \ref{uniq-cont}. Finally, transitivity follows from Remark \ref{trans-prop}.
\end{proof}

Note that bs-stability only mentions basic 1-types, so it transfers immediately.  Thus, the only property left is symmetry, which is discussed in the next two sections.

We conclude by proving a concatenation lemma for independent sequences. This is already proved for good frames in \cite[Proposition 4.1]{jasi}, but the proof relies on \cite[Proposition 2.6]{jasi}, which is proved as \cite[Proposition 3.1.8]{jrsh875} and uses symmetry in an essential way. Here, we improve this to just requiring that $\s$ is a pre-frame that also satisfies amalgamation, existence, continuity, and transitivity. In particular, we avoid any use of symmetry or nonforking amalgamation. This shows that the situation is somewhat similar to the first-order context, where concatenation holds in any theory (see, e.g.,\ \cite[Lemma 1.6]{primer}).

\begin{thm}[Concatenation]\label{concat}
Assume $\s$ is a pre-$\F$-frame with amalgamation, existence, transitivity, and continuity.  Let $M \lea M_0 \lea M_1 \lea M_2$ be such that $\ba = \seq{a_i:i < \alpha}$ is independent from $M_0$ over $M$ in $M_1$ and $\bb = \seq{b_i : i < \beta}$ is independent from $M_1$ over $M$ in $M_2$.  Then $\ba\bb$ is independent from $M_0$ over $M$ in $M_2$.
\end{thm}
\begin{proof}
From the independence of $\ba$ from $M_0$ over $M$ in $M_1$, there is a continuous, increasing $\seq{M_0^i : i \leq \alpha}$ and $N_0^+$ such that
\begin{itemize}
	\item $M_0 \lea M_0^i \lea N_0^+$;
	\item $M_1 \lea N_0^+$; and
	\item $\nfs{M}{a_i}{M_0^i}{M_0^{i+1}}$.
\end{itemize}

From the independence of $\bb$ from $M_1$ over $M$ in $M_2$, there is a continuous, increasing $\seq{M_1^i : i \leq \beta}$ and $N_1^+$ such that
\begin{itemize}
	\item $M_1 \lea M_1^i \lea N_1^+$;
	\item $M_2 \lea N_1^+$; and
	\item $\nfs{M}{b_i}{M_1^i}{M_1^{i+1}}$.
\end{itemize}

Define increasing and continuous $\seq{N_1^i : i \leq \beta}$ and $\seq{g_i : M_1^i \to N_i : i \leq \beta}$ such that:
\begin{itemize}
	\item  $N_0^+ \prec N_1^0$ and $g_0 \rest M_1 = \id_{M_1}$; and
	\item For all $i < \beta$, $g_{i+1}(b_i) \nf_M^{N_1^{i+1}} N_1^i$.
\end{itemize}

This can easily be constructed by inductions: amalgamate $M_1^0$ and $N_0^+$ over $M_1$ to get $N_1^0$ and $g_0$.  At successor steps, apply Lemma \ref{ext-vars} and take unions at limit stages.

After this construction, amalgamate $N_1^+$ and $N_1^\beta$ over $M_1^\beta$ to get $N^{++}$ and $g$ so the following diagram commutes for $j < \beta$:
\[
\xymatrix{
& & & N_0^+ \ar@{.>}[r] & N_1^0 \ar@{.>}[r]& N_1^j \ar@{.>}[r] & N_1^{j+1} \ar@{.>}[r] & N^{++}\\
 & M_0^0 \ar[r] & M_0^\alpha \ar[ur]& & M_1^0 \ar@{.>}[u]_{g_0} \ar[r] & M_1^j \ar@{.>}[u]_{g_j} \ar[r] & M_1^{j+1} \ar@{.>}[u]_{g_{j+1}} \ar[r] & N_1^+ \ar@{.>}[u]_g\\
M \ar[r] & M_0 \ar[u] \ar[rr] & & M_1 \ar[uu] \ar[ur] \ar[rrrr] & & & & M_2 \ar[u]
}
\]

Define the sequence $\seq{N^i : i \leq \alpha + \beta}$ by
$$N^i := \begin{cases} M_0^i &\mbox{if } i < \alpha \\ 
N_1^j & \mbox{if }  i = \alpha + j \in [\alpha, \beta] \end{cases} $$

{\bf Claim:} This sequences witnesses that $\bc := \ba^\frown g(\bb)$ is independent from $M_0$ over $M$ in $N^{++}$.

{\bf Proof of Claim:}  It is easy to see that this sequence is of the proper type, i.e.\ it is increasing and continuous and $M_0 \lea N^i \lea N^{++}$ for all $i \le \alpha + \beta$.

If $i < \alpha$, then we need to show that $\nfs{M}{c_i}{N^i}{N^{i+1}}$, which is the same as $\nfs{M}{a_i}{M_0^i}{M_0^{i+1}}$. This just follows from independence of $\ba$.


If $i = \alpha + j \ge \alpha$, then we need to show that $\nfs{M}{c_i}{N^i}{N^{i+1}}$, which is the same as $\nfs{M}{g_{j + 1} (b_j)}{N_1^j}{N_1^{j+1}}$.  This holds directly by the construction. \hfill $\dag_{{\bf Claim}}$\\

Notice that the map $g$ shows that $\tp(\ba g(\bb)/M_0; N_1^\beta) = \tp(\ba\bb/M_0; M_2)$.  Thus, by Invariance (Lemma \ref{type-indep}), we have that $\ba\bb$ is independent from $M_0$ over $M$ In $M_2$.
\end{proof}



\section{Symmetry in long frames}\label{sym-long-frames}

In this section, we discuss when symmetry transfers from a good frame to its elongation. We do so by studying the following unordered version of independence:

\begin{defin}
  A set $I$ is said to be \emph{independent in $(M, M_0, N)$} if \emph{some} enumeration of $I$ is independent in $(M, M_0, N)$. As usual, we say instead that $I$ is independent from $M_0$ over $M$ in $N$.
\end{defin}

\subsection{Several versions of continuity}\label{continuity-subsec}

The notion of a set being independent gives rise to several notions of continuity. We gave a definition of continuity for a pre-frame $\s$, as well as continuity for the corresponding frame of independent sequences $\s^{<\theta_{\s}}$ (what Jarden calls the continuity of serial independence \cite[Definition 5.3]{jarden-tameness-apal}, see Remark \ref{jarden-serial-rmk}). We can now study the corresponding continuity properties for sets rather than sequences: for $\s$ a pre-$\F$-frame, let us say that $\s^{<\theta_{\s}}$ has the \emph{unordered continuity property} if for every increasing chain $\seq{M_{\alpha} : \alpha < \delta}$ every $N$ containing $\bigcup_{\alpha < \delta} M_\alpha$ and every $I \subseteq |N|$, $I$ is independent from $\bigcup_{\alpha < \delta} M_\alpha$ over $M_0$ if $I$ is independent from $M_\alpha$ over $M_0$ for all $\alpha < \delta$ (so the enumeration witnessing the independence is allowed to change each time). Confusingly, Jarden and Sitton \cite[Definition 5.5]{jasi} call this the continuity property.

Another notion of continuity was also introduced by Jarden and Sitton. Let us say that a set $I$ is \emph{finitely independent} (from $M_0$ over $M$ in $N$) if every finite subset of $I$ is. Jarden and Sitton \cite[Definition 8.2]{jasi} say that the finite continuity property holds when unordered continuity holds for the notion of finite independence. We will refer to this as \emph{unordered finite continuity}.

Jarden and Sitton show \cite[Proposition 8.4]{jasi} that unordered finite continuity holds in $\goodms{St}$ frames which satisfy the additional assumption of the conjugation property and being weakly successful. Using the (ordered) continuity property for independent sequences (Corollary \ref{uniq-cont}), together with Fact \ref{good-frame-permut} below, we immediately obtain that the unordered finite continuity holds in any $\goodms{St}$ frame.

\begin{fact}[Theorem 4.2.(a) in \cite{jasi}]\label{good-frame-permut}
  Let $\s$ be a $\goodms{St}$ $\F$-frame. If $\bar{a}$ is a finite tuple independent from $M'$ over $M$ in $N$, then any permutation of $\bar{a}$ is independent from $M'$ over $M$ in $N$.
\end{fact}

Implicit in this notion is a notion of independence being \emph{finitely witnessed} \cite[Definition 3.4]{jasi} which says that a set $I$ is independent if and only if all its finite subsets are. We give a more general parametrized definition here:

\begin{defin}\label{strong-continuity}
  Let $\s$ be a pre-$\F$-frame and $\mu \le \theta_\s$ be a cardinal. We say that \emph{$\mu$-independence in $\s$ is finitely witnessed} if for any $M_0 \lea M \lea N$ in $K_{\F}$ and any $I \subseteq N$ with $|I| < \mu$, $I$ is independent from $M$ over $M_0$ in $N$ if and only if all its finite subsets are independent from $M$ over $M_0$ in $N$.
  
  If $\mu = \theta_\s$, we omit it.
\end{defin}
\begin{remark}
  In \cite[Theorem 9.3]{jasi} shows that independence is finitely witnessed in a good $\lambda$-frame assuming the conjugation property, categoricity in $\lambda$, and density of uniqueness triples. Earlier, Shelah had proven the same result under stronger hypotheses \cite[Theorem III.5.4]{shelahaecbook}. 
\end{remark}
\begin{remark}\label{jasi-rmk}
  It is straightforward to see that if independence is finitely witnessed and the finite unordered continuity property holds, then the unordered continuity property holds. Recall from the discussion above that the finite unordered continuity property holds in any $\goodms{St}$-frame.
\end{remark}

Our next goal is to show that if $\s^{<\mu}$ has symmetry then $\mu$-independence is finitely witnessed (Theorem \ref{sym-strong-cont}). Together with Lemma \ref{sym-elongation} deducing symmetry from the frame being sufficiently global, this will show (Corollary \ref{cor-summary}) that tameness implies independence is finitely witnessed.

\subsection{Symmetry implies being finitely witnessed}

First we show that symmetry is equivalent to showing that the order of enumeration does not matter. The finite case is essentially Fact \ref{good-frame-permut}. To state the infinite case precisely, we introduce new terminology:

\begin{defin}\label{sym-indep}
  Let $\s$ be a pre-$\F$-frame and $\mu \le \theta_\s$ be a cardinal. We say that $\s$ has \emph{$\mu$-symmetry of independence} if for any $M_0 \lea M \lea N$ in $K_{\F}$ and any $I \subseteq N$ with $|I| < \mu$, $I$ is independent from $M$ over $M_0$ in $N$ if and only if \emph{every} enumeration of $I$ is independent from $M$ over $M_0$ in $N$.

  If $\mu = \theta_\s$, we omit it.
\end{defin}

Thus a restatement of Fact \ref{good-frame-permut} is that any $\goodms{St}$ frame has $\aleph_0$-symmetry of independence. The next theorem says that $\mu$-symmetry of independence is equivalent to $\s^{<\mu}$ having symmetry.

\begin{thm}\label{sym-equiv}
  Let $\s$ be a $\goodm$ $\F$-frame and let $\mu \le \theta_\s$ be a cardinal. The following are equivalent:

  \begin{enumerate}
    \item $\s^{<\mu}$ has symmetry.
    \item For any $M_0 \lea M \lea N$ in $K_{\F}$ and $\ba \bb \in N$ such that $\ell (\ba \bb) < \mu$, $\ba \bb$ is independent from $M$ over $M_0$ in $N$ if and only if $\bb \ba$ is independent in from $M$ over $M_0$ in $N$.
    \item $\s$ has $\mu$-symmetry of independence.
  \end{enumerate}
\end{thm}
\begin{proof}
  We first show (1) is equivalent to (2). Assume $\s^{<\mu}$ has symmetry, and let $M_0 \lea M \lea N$ in $K_{\F}$ and $\ba \bb \in N$ be such that $\ell (\ba \bb)  < \mu$ and $\ba \bb$ is independent from $M$ over $M_0$ in $N$. Then there exists $\seq{M^i : i \le \ell(\ba\bb)}$ and $N^+ \gea N$ witnessing it. Say $\alpha := \ell (\ba)$. Then $\ba \in M^\alpha$, $tp(\ba/M; M^\alpha) \in \Sabs{\alpha} (M_0)$, and $\bb$ is independent from $M^\alpha$ over $M$ in $N^+$, i.e.\ $\nfs{M}{\bb}{M^\alpha}{N^+}$. By Symmetry, there must exist a model $M'$ containing $\bb$ and $N^{++} \gea N^+$ such that $\nfs{M}{\ba}{M'}{N^{++}}$. By Monotonicity, $\nfs{M_0}{\ba}{M}{N^{++}}$, so by Transitivity, $\nfs{M_0}{\ba}{M'}{N^{++}}$. By Monotonicity, $\nfs{M_0}{\bb}{M}{M'}$. By concatenation (Theorem \ref{concat}), $\nfs{M_0}{\bb \ba}{M}{N^{++}}$ and so by Monotonicity, $\nfs{M_0}{\bb \ba}{M}{N}$, as needed. Conversely, assume (2). Assume $\nfs{M_0}{\ba_1}{M_2}{N}$ with $\ba_1 \in \fct{<\mu}{N}$, and $\ba_2 \in \fct{<\mu}{M_2}$ is such that $\tp (\ba_2 / M_0; N) \in \Sabs{<\mu} (M_0)$. By existence, $\nfs{M_0}{\ba_2}{M_0}{M_2}$. By concatenation, $\nfs{M_0}{\ba_1 \ba_2}{M_0}{N}$. By (2), $\nfs{M_0}{\ba_2 \ba_1}{M_0}{N}$. By definition of independent, there exists $M_1$ containing $\ba_1$ and $N' \gea N$ such that $\nfs{M_0}{\ba_2}{M_1}{N'}$, as needed.

  Next, we show that (2) is equivalent to (3). It is clear that (3) implies (2), so we assume (2) and we prove (3) as follows: we prove the following by induction on $\alpha < \mu$:
  
\begin{enumerate}
\item[$(*)_\alpha$] Let $M_0 \lea M \lea N$ be in $K_\mathcal{F}$ and let $I \subseteq |N|$ have size less than $\mu$.  If $I$ is independent from $M$ over $M_0$ in $N$, then \emph{every} enumeration of $I$ \emph{of order type $\alpha$} is independent from $M$ over $M_0$ in $N$.
\end{enumerate}
 
So let $\alpha < \mu$ and assume $(*)_\beta$ holds for all $\beta < \alpha$. Suppose $I$ as above is independent from $M$ over $M_0$ in $N$ and let $\seq{a_i : i < \alpha}$ be an enumeration of $I$ of type $\alpha$.

 First, suppose $\alpha$ is finite. Then $I$ is finite so Fact \ref{good-frame-permut} gives the result.
 
 Second, suppose $\alpha = \beta+1$ is an infinite successor. Then $\seq{a_\beta} ^\frown \seq{a_i : i < \beta}$ has order type $\beta$ and so (by $(*)_\beta$) is independent from $M$ over $M_0$ in $N$. Since (2) implies (1), the original sequence must also be independent.
 
 Finally, suppose that $\alpha$ is limit.  By monotonicity, every subset of $I$ is independent from $M$ over $M_0$ in $N$. In particular, for each $\beta < \alpha$ $\{a_i : i < \beta\}$ is independent from $M$ over $M_0$ in $N$, and so by $(*)_\beta $ $\seq{a_i : i < \beta}$ is also independent from $M$ over $M_0$ in $N$.  Thus by continuity (Corollary \ref{uniq-cont}) $\seq{a_i : i < \alpha}$ is independent from $M$ over $M_0$ in $N$.
  
\end{proof}

As a corollary, we manage to solve Exercise III.9.4.1 in \cite{shelahaecbook}:

\begin{cor}
  Let $\s$ be a good [$\goodms{St}$] $\F$-frame. Then $\s^{<\omega}$ is a good [$\goodms{St}$] $\F$-frame.
\end{cor}
\begin{proof}
  By Corollary \ref{tuple-transfer}, $\s^{<\omega}$ is a $\goodm$ $\F$-frame. By Fact \ref{good-frame-permut}, $\s$ has $\aleph_0$-symmetry of independence. By Theorem \ref{sym-equiv}, $\s^{<\omega}$ has symmetry, as needed.  Since bs-stability only refers to basic 1-types, $\s$ satisfies it if and only if $\s^{<\omega}$ does.
\end{proof}

Unfortunately, we do not know whether in general $\omega$ above can be replaced by a larger ordinal. To give a criteria on when this is possible, we show that independence being finitely witnessed (see Definition \ref{strong-continuity}) follows from symmetry.

\begin{thm}\label{sym-strong-cont}
  Let $\s$ be a $\goodms{St}$ $\F$-frame and let $\mu \le \theta_\s$ be a cardinal. If $\s^{<\mu}$ has symmetry, then $\mu$-independence in $\s$ is finitely witnessed.
\end{thm}
\begin{proof}
  By Theorem \ref{sym-equiv} $\s$ has $\mu$-symmetry of independence, and by Corollary \ref{uniq-cont} $\s^{<\mu}$ has continuity. Let $M_0 \lea M \lea N$ be in $K_{\F}$ and let $I \subseteq N$ be such that $|I| < \mu$. Assume that every finite subset of $I$ is independent in from $M$ over $M_0$ in $N$. Assume inductively that $\mu_0$-independence is finitely witnessed for all $\mu_0 < \mu$. Let $\mu_0 := |I|$ and write $\{a_i : i < \mu_0\}$. Let $I_i := \{a_i : j < i\}$. By the induction hypothesis, $I_i$ is independent from $M$ over $M_0$ in $N$ for all $i < \mu_0$. By $\mu$-symmetry of independence, the ordered sequence $\seq{a_j : j < i}$ is independent from $M$ over $M_0$ in $N$. By continuity of $\s^{<\mu}$, $\seq{a_i : i < \mu_0}$ is independent from $M$ over $M_0$ in $N$. Thus $I$ is independent from $M$ over $M_0$ in $N$, as desired.
\end{proof}
\begin{remark}
  A similar proof shows that the \emph{ordered} version of $\mu$-independence being finitely witnessed (that is, a sequence is independent if and only if all of its finite subsequences are) is \emph{equivalent} to symmetry in $\s^{<\mu}$.
\end{remark}

Next, we show symmetry indeed holds in the elongation if the original frame is ``sufficiently global'' (this does not even use that $\s$ has symmetry):

\begin{lem}\label{sym-elongation}
  Assume $\s$ is a $\goodm$ $\F$-frame and $\F = [\lambda, \theta)$. If $\theta \ge \beth_{\left(2^{\lambda}\right)^+}$, then $\s_\lambda^{<\lambda^+}$ has symmetry.
\end{lem}
\begin{proof}
Using uniqueness and local character, it is straightforward to see that $K_{\F}$ is stable in $2^\lambda$ (for $1$-types), see e.g.\ \cite[Proposition 6.4]{ss-tame-jsl}. By Fact \ref{stab-longtypes} this means $K_{\F}$ is stable in $2^\lambda$ for $\lambda$-types. Then the same nonstructure proof as \cite[Corollary 6.11]{ss-tame-jsl} generalizes: if $\s$ does not have symmetry, then the same proof as \cite[Theorem 5.14]{bgkv-apal} shows that $K_{\F}$ has an order property, and this order property is enough to deduce instability in $2^\lambda$ for $\lambda$-types (see \cite[Section 4]{sh394} or \cite[Fact 5.13]{bgkv-apal} for a sketch).
\end{proof}

Note, by uniqueness and local character, if $\chi := \mathfrak{tb}^1_\lambda := \sup_{M \in K_\lambda} |\Ss(M)|$, and $\s$ is a $\goodm$ $[\lambda, \chi]$-frame, then $\s_\chi$ will satisfy bs-stability (and hence be a $\goodms{S}$-frame); see \cite[Proposition 6.4]{ss-tame-jsl}.

We now apply the lemma to the maximal elongation of a $(\ge \lambda)$-frame $\s$, namely $\s^{<\infty} := \cup_{\alpha \in \textbf{ON}} \s^{<\alpha}$.

\begin{cor}\label{sym-elong-infty}
  Assume $\s$ is a $\goodm$ $(\ge \lambda)$-frame. Then $\s^{<\infty}$ has symmetry.
\end{cor}
\begin{proof}
  Use Lemma \ref{sym-elongation} with each $\lambda' \in [\lambda, \infty)$.
\end{proof}
\begin{cor}\label{infty-good-frame}
  Assume $\s$ is a $\goodms{S}$ [$\goodm$] $(\ge \lambda)$-frame. Then $\s^{<\infty}$ is a good [$\goodms{St}$] $(<\infty, \ge \lambda)$-frame.
\end{cor}
\begin{proof}
  Combine Corollary \ref{tuple-transfer} and Corollary \ref{sym-elong-infty}.
\end{proof}

\section{Applications}\label{going-up}

This section gives some applications of these results.

\subsection{Dimension}

In \cite[Definition III.5.12]{shelahaecbook}, Shelah introduced a notion of dimension based on a frame.  In \cite[Conclusion III.5.14]{shelahaecbook}, he shows that this notion is well-behaved (in the sense of Corollary \ref{good-dimension}) from an assumption that is a little stronger than $\s$ being weakly successful and Jarden and Sitton \cite[Theorem 1.1]{jasi} reduce this assumption to just assuming the $\goodms{St}$ $\lambda$-frame has the unordered continuity property.  A corollary of our results on symmetry and independence being finitely witnessed is that we can remove any extra hypothesis.

\begin{cor}\label{good-dimension}
  Let $\s$ be a $\goodms{St}$ $\lambda$-frame and assume $\s^{<\lambda^+}$ has symmetry. Let $M \lea M_0 \lea N$ be in $K_\lambda$. If:

  \begin{enumerate}
    \item $P \subseteq \Sbs (M_0)$
    \item $I_1, I_2$ are each $\subseteq$-maximal sets in

      $$\{I : I \text{ is independent from }M_0 \text{ over } M \text{ in } N \text{ and } a \in I \Rightarrow \tp (a / M_0; N) \in P\}$$
    \item One of $I_1$, $I_2$ is infinite.
  \end{enumerate}

  Then $I_1$ and $I_2$ are both infinite and $|I_1| = |I_2|$.
\end{cor}
\begin{proof}
Since Symmetry holds, independence in $\s$ is finitely witnessed by Theorem \ref{sym-strong-cont}. Recalling Remark \ref{jasi-rmk}, the hypotheses of \cite[Theorem 1.1]{jasi} hold, and the conclusion is this result.
\end{proof}

This dimension--defining $\dim(P, N)$ to be the single infinite size of a $I_1$ from Corollary \ref{good-dimension}--is used to develop the theory of regular types in \cite[Section III.10]{shelahaecbook}.  As it stands, there is no known example showing that symmetry is necessary to develop a dimension theory (or a theory of regular types). In fact, there is no known example of a $\goodm$-frame which fails to have symmetry (i.e.\ it is not known whether symmetry follows from the other axioms of a good frame, although we suspect it does not). However, the fact that this definition compares independent sets rather than sequences implicitly assumes the symmetry of independence (see Theorem \ref{sym-equiv}).

\subsection{Tameness and extending frames revisited}

Recall the definition of tameness from Definition \ref{tameness-def}.  The first author \cite{ext-frame-jml} first studied the connection between tameness and frames. As in \cite[Theorem 3.2]{ext-frame-jml}, having a frame that spans multiple cardinals already gives some tameness.

\begin{prop}\label{tameness-uniq}
  Assume $\s := (K, \nf, \Sbs)$ is a $\goodm$ $\F$-frame. Let $\F := [\lambda, \theta)$. 

    For each $\alpha < \theta$, $K$ is $(\lambda + |\alpha|, <\theta)$-tame for the basic types of $\s^{<\theta}$ of length $\le \alpha$.
\end{prop}
\begin{proof}
  Let $\alpha < \theta$, and let $p, q \in \Sabs{\le \alpha} (M)$ be distinct. By the moreover part of Theorem \ref{lc-ext-transfer}.(\ref{lc-transfer}), one can find $M_0 \lea M$ in $K_{\le \lambda + |\alpha|}$ such that both $p$ and $q$ do not fork over $M_0$. By uniqueness, we must have $p \upharpoonright M_0 \neq q \upharpoonright M_0$, as needed.
\end{proof}

In \cite{ext-frame-jml}, the main concern was using $\lambda$-tameness to extend a $\lambda$-frame to a $(\geq \lambda)$-frame.  The definition of the extension and the preservation of several properties were already done by Shelah.

\begin{defin}[Going up, Definitions II.2.4 and II.2.5 of \cite{shelahaecbook} ] \label{going-up-def}
  Let $\s := (K, \nf, \Sbs)$ be a pre-$(<\alpha, \lambda)$-frame, and let $\F = [\lambda, \theta)$ be an interval of cardinals as usual. Define $\s_{\F} := (K, \nf_{\F}, \Sbs_{\F})$ as follows:

  \begin{itemize}

    \item For $M_0 \lea M_1 \lea N$ in $K_{\F}$ and $\ba \in \fct{<\alpha}{N}$, $\nf_{\F}(M_0, M_1, \ba, N)$ if and only if there exists $M_0' \lea M_0$ in $K_\lambda$ such that for all $M_0' \lea M_1' \lea N' \lea N$ with $\ba \in N'$, and $M_1'$, $N'$ in $K_\lambda$, we have $\nfs{M_0'}{\ba}{M_1'}{N'}$.
    \item For $M \in K_{\F}$ and $p \in \Ss^{<\alpha} (M)$, $p \in \Sbs_{\F} (M)$ if and only if there exists $N \gea M$ and $\ba \in N$ such that $p = \tp (\ba / M; N)$ and $\nf_{\F} (M, M, \ba, N)$.
  \end{itemize}
\end{defin}

\begin{fact} \label{basic-frame-fact}
  Let $\s$ be a $\goodm$ $\lambda$-frame, and let $\F = [\lambda, \theta)$ be an interval of cardinals as usual. Then $\s_{\F}$ satisfies all the properties of a good $\F$-frame except perhaps bs-stability, existence, uniqueness, and symmetry.
\end{fact}
\begin{proof}
See \cite[Section II.2]{shelahaecbook}.
\end{proof}

Transferring the rest of the properties from a good $\lambda$-frame to a good $[\lambda, \lambda^+]$-frame was the project of the rest of \cite[Section II]{shelahaecbook} and involved combinatorial set-theoretic hypotheses and shrinking the AEC under consideration.  \cite{ext-frame-jml} replaced these assumptions with tameness.

\begin{fact}[Theorem 8.1 in \cite{ext-frame-jml}]\label{will-transfer}
  Let $\s$ be a $\goodm$ [$\goodms{S}$] $\lambda$-frame, and let $\F = [\lambda, \theta)$ be an interval of cardinals where $\theta > \lambda$ can be $\infty$. If $K_{\F}$ has amalgamation and no maximal models, the following are equivalent:

    \begin{enumerate}
      \item $K$ is $\lambda$-tame for the basic types of $\s_{\F}$.
      \item $\s_{\F}$ is a $\goodm$ [$\goodms{S}$] $\F$-frame.
    \end{enumerate}

    Moreover, if $\s$ has symmetry and $K$ is $(\lambda, \theta)$-tame for $2$-length types, then $\s_{\F}$ has symmetry. In this case, the no maximal models hypothesis is not needed.
\end{fact}

A surprising feature of this result is that, although the frames involved only $1$-types, the proof required tameness for longer types.  This is connected to an emerging divide in the literature on tameness: although Grossberg and VanDieren's initial definition for tameness \cite{tamenessone} included the length of types, their categoricity transfer \cite{tamenesstwo, tamenessthree} and several subsequent works (e.g.\ \cite{b-k-vd-spectrum, liebermanrank})) required only tameness for 1-types.  However, later works, beginning with Boney and Grossberg \cite{bg-v10} and Vasey \cite{indep-aec-apal} (begun after the initial submission of this paper), began to use tameness for longer types (and stronger locality properties like type shortness) in essential ways.  It remains to be seen which version of tameness is the ``proper one'' for developing classification theory (or indeed if they are the same under some reasonable hypothesis).  However, Fact \ref{will-transfer} seemed to straddle this divide: it used more than tameness for 1-types, but not much more and it was unclear if the use was essential.

By using the results of this paper, we are able to remove the assumption of tameness for 2-types in the proof of symmetry and show that the use was unnecessary.  We know that the tameness for 1-types gives uniqueness for the extension $\s_{\F}$, and that this uniqueness transfers to uniqueness for the elongation of $\s_{\F}$.  Thus, it suffices to show that the 2-types considered in the proof of symmetry are basic in this sense, which we do in Theorem \ref{sym-abstract-transfer}.  Before we do this, we must be careful that the order does not matter, i. e., that extending and then elongating a frame gives you the same result as elongating and then extending it.  One direction is easy.

\begin{prop}\label{partial-commut}
  Let $\s := (K, \nf, \Sbs)$ be a pre-$\lambda$-frame, and let $\F := [\lambda, \theta)$ be an interval of cardinals as usual. Assume $K_{\F}$ has amalgamation. Then:

    $$
    \left(\s_{\F}\right)^{<\lambda^+} \subseteq \left(\s^{<\lambda^+} \right)_{\F}
    $$

    Where $\subseteq$ is taken componentwise.
\end{prop}
\begin{proof}
  Assume we know that $\nf_{(\s_{\F})^{<\lambda^+}} (M_0, M, \ba, N)$. We show that $\nf_{(\s^{<\lambda^+})_{\F}} (M_0, M, \ba, N)$. The proof of inclusion of the basic types is completely similar.

  Let $\ba := \seq{a_i : i < \beta}$, for $\beta < \lambda^+$. By assumption, $\ba$ is independent (with respect to $\nf_{\F}$) from $M$ over $M_0$ in $N$. Fix $\seq{M^i : i \le \beta}$ and $N^+$ witnessing the independence. In particular, for every $i < \beta$, $\nf_{\F} (M_0, M^i, a_i, N^+)$. By definition of $\nf_\F$, this implies in particular that for each $i < \beta$, there exists $M_i^0 \lea M_0$ in $K_\lambda$ such that $\nf_{\F} (M_i^0, M^i, a_i, N^+)$. Using the Löwenheim-Skolem axiom and the fact that $|\beta| \le \lambda$, we can choose $M^* \lea M_0$ in $K_\lambda$ such that for all $i < \beta$,we have $M_i^0 \lea M^*$.  Thus, $\nf_{\F} (M^0, M_i, a_i, N^+)$ for all $i < \beta$. In particular, $\ba$ is independent (with respect to $\nf_{\F}$) from $M$ over $M^*$ in $N$.

  Now fix any $M_0', N' \in K_\lambda$ such that $\ba \in N'$, $M^* \lea M_0' \lea M$, and $M_0'  \lea N' \lea N$. We claim that $\ba$ is independent (with respect to $\nf$) from $M_0'$ over $M^*$ in $N'$, i.e.\ $\nf^{<\lambda^+} (M^*, M_0', \ba, N')$. To see this, construct $\seq{M_i' \in K_\lambda : i \le \beta}$ increasing continuous such that for all $i \le \beta$, $M^* \lea M_i' \lea M^i$ and $a_i \in M_{i + 1}'$. Finally, pick $(N^+)' \in K_\lambda$ such that $M_\beta', N' \lea (N^+)' \lea N^+$. Then $\seq{M_i' : i \le \beta}$ and $(N^+)'$ witness our claim. By definition, this means exactly that $\nf_{(\s^{<\alpha})_{\F}} (M_0, M, \ba, N)$, as needed.
\end{proof}

The converse needs more hypotheses and relies on Corollary \ref{tuple-transfer}:

\begin{thm}\label{full-commut}
  Let $\s := (K, \nf, \Sbs)$ be a $\goodm$ $\lambda$-frame, and let $\F := [\lambda, \theta)$ be an interval of cardinals as usual. Assume that $\s_\F$ is a $\goodm$ $\F$-frame. Then: 

    $$
    \left(\s_{\F}\right)^{<\lambda^+} = \left(\s^{<\lambda^+}\right)_{\F}
    $$
\end{thm}
\begin{proof}
  By Proposition \ref{partial-commut} and existence, it is enough to show $\nf_{(\s^{<\lambda^+})_{\F}} \subseteq \nf_{(\s_{\F})^{<\lambda^+}}$. Assume $\nf_{(\s^{<\lambda^+})_{\F}} (M, N, \ba, \bigN)$. By definition of $\nf_{(\s^{<\alpha})_{\F}}$ and monotonicity, we can assume without loss of generality that $M \in K_\lambda$. We know that for all $N' \lea N$ and $\bigN' \lea \bigN$ in $K_\lambda$ with $M \lea N \lea \bigN'$ and $\ba \in \bigN'$, $\ba$ is independent (with respect to $\nf$) from $N'$ over $M$ in $\bigN'$. We want to see that $\ba$ is independent (with respect to $\nf_{\F}$) from $N$ Over $M$ in $\bigN$.

Let $\mu \ge \lambda$ be such that $N, \bigN \in K_{\le\mu}$. Work by induction on $\mu$. We already have what we want if $\mu = \lambda$, so assume $\mu > \lambda$. Let $(N_i)_{i \le \mu}$ be an increasing continuous resolution of $N$ such that $N_\mu = N$, $N_0 = M$, $\|N_i\| = \lambda + |i|$.

By the induction hypothesis and monotonicity, $\ba$ is independent (with respect to $\nf_{\F}$) from $N_i$ over $M$ in $\bigN$ for all $i < \mu$. In other words, for any $i < \mu$, $\tp (\ba / N_i; \bigN))$ does not fork (in the sense of $\left(\s_{\F}\right)^{<\lambda^+}$) over $M$. By Corollary \ref{tuple-transfer}, we know that $\left(\s_{\F}\right)^{<\lambda^+}$ has continuity. Thus $\tp (\ba / N; \bigN)$ also does not fork (in the sense of $\left(\s_{\F}\right)^{<\lambda^+}$) over $M$. This is exactly what we needed to prove.
\end{proof}

We can now prove an abstract symmetry transfer that does not mention tameness.

\begin{thm}\label{sym-abstract-transfer}
  Assume $\s$ is a $\goodm$ $\F$-frame. Let $\F := [\lambda, \theta)$.

    Then $\s$ has symmetry if and only if $\s_\lambda$ has symmetry.
\end{thm}
\begin{proof}
  Of course, symmetry for $\s$ implies in particular symmetry for $\s_\lambda$. Now assume symmetry for $\s_\lambda$.

  First note that $\s = (\s_\lambda)_{\F}$. This is because by the methods of \cite[Section II.2]{shelahaecbook} (see especially Claim 2.14 and the remark preceding it), there is at most one $\goodm$ $\F$-frame extending $\s_\lambda$, and it is given by $(\s_\lambda)_{\F}$ if it exists.

  Let $\ts := \s_\lambda := (K, \nf, \Sbs)$. Thus $\s = \ts_\F$. Recall that \cite[Theorem 6.1]{ext-frame-jml} proves symmetry for $\s$ assuming $(\lambda, <\theta)$-tameness for $2$-types. We revisit this proof and use the same notation. 

Suppose $\nf_{\F} (M_0, M_2, a_1, M_3)$, $a_2 \in M_2$ with $\tp (a_2 / M_0; M_3) \in \Sbs_{\F} (M_0)$. Let $M_0 \lea M_1 \lea M_3$ be a model containing $a_1$. By existence, there is $M_3' \gea M_3$ and $a' \in M_3'$ such that $\nf_{\F} (M_0, M_1, a', M_3')$ and $\tp (a' / M_0; M_3') = \tp (a_2 / M_0; M_3)$. Boney argues it is enough to see that $p := \tp (a_1 a_2 / M_0; M_3) = \tp (a_1 a' / M_0; M_3') =: p'$, shows that this equality holds for all restrictions to models of size $\lambda$, and then uses tameness for 2-types. This is not part of our hypotheses, but by Proposition \ref{tameness-uniq}, it is enough to see that $p, p'$ are basic types of $\s^{\le 2}$. 

First, let us see that $a_1 a_2$ is independent (with respect to $\nf_{\F}$) from $M_0$ over $M_0$ in $M_3$. The increasing chain $(M_0, M_2, M_3)$ witnesses that $a_2a_1$ is independent (with respect to $\nf_{\F}$ again) from $M_0$ over $M_0$ in $M_3$. Thus $\tp (a_2 a_1 / M_0; M_3) \in \Sbs_{\s^{\le 2}} (M_0)$, and $\s^{\le 2} = \left(\ts_{\F}\right)^{\le 2} = \left(\ts^{\le 2}\right)_{\F}$ by Theorem \ref{full-commut}. Thus there exists $M_0' \lea M_0$ in $K_\lambda$ such that for all $M_0'' \gea M_0'$ in $K_\lambda$ with $M_0'' \lea M_3$, $\tp (a_2 a_1 / M_0''; M_3)$ does not fork (in the sense of $\ts^{\le 2}$) over $M_0'$. Since we have symmetry in $\ts$, we have (by Fact \ref{good-frame-permut}) that also $\tp (a_1 a_2 / M_0'' ; M_3)$ does not fork over $M_0'$ for all $M_0'' \gea M_0'$, $M_0'' \lea M_3$ in $K_\lambda$. Thus by definition and Theorem \ref{full-commut} again, $a_1a_2$ is independent (with respect to $\nf_{\F}$) from $M_0$ over $M_0$ in $M_3$, as needed. Similarly, $(M_0, M_1, M_3')$ witnesses that $a_1 a'$ is independent from $M_0$ over $M_0$ in $M_3'$. Thus $p$ and $p'$ are basic types of $\s^{\le 2}$, as needed.
\end{proof}

We can now prove the desired improvement.

\begin{cor}\label{will-transfer-improved}
  Let $\s := (K, \nf, \Sbs)$ be a good $\lambda$-frame, and let $\F := [\lambda, \theta)$ be an interval of cardinals, where $\theta > \lambda$ is either a cardinal or $\infty$. Assume $K_{\F}$ has amalgamation and $K$ is $(\lambda, < \theta)$-tame. Then $\s_{\F}$ is a good $\F$-frame. 
\end{cor}
\begin{proof}
  By the proof of Fact \ref{will-transfer}, $\s_{\F}$ has all the properties of a good frame, except perhaps no maximal models and symmetry. Symmetry follows from the previous theorem and \cite[Theorem 7.1]{ext-frame-jml} now gives us no maximal models.
\end{proof}

While we were writing up this paper, Adi Jarden \cite{jarden-tameness-apal} independently gave this improvement, with the additional hypothesis that the frame was weakly successful (which he used to get the $\lambda^+$-continuity of serial independence property; see Remark \ref{jarden-serial-rmk}).

\subsection{Conclusion}

We conclude by summarizing what our results give from a good frame, amalgamation, and tameness:

\begin{cor}\label{cor-summary}
  Let $\s := (K, \nf, \Sbs)$ be a good $\lambda$-frame. If $K_{\ge \lambda}$ has amalgamation and is $\lambda$-tame, then:

  \begin{enumerate}

    \item $\s_{\ge \lambda}$ is a good $(\ge \lambda)$-frame, and in fact even $\ts := \left(s_{\ge \lambda}\right)^{<\infty}$ is a good $(<\infty, \ge \lambda)$-frame.
    \item For all $\alpha$, $K$ is $(\lambda + |\alpha|)$-tame for the basic types of $\ts$ of length $\le \alpha$.
    \item $\left(\s^{<\lambda^+}\right)_{\ge \lambda} = \left(\s_{\ge \lambda}\right)^{<\lambda^+}$.
    \item $\ts$ has symmetry of independence and independence in $\s_{\ge \lambda}$ is finitely witnessed. 
    \item We have a well-behaved notion of dimension: For $M \lea M_0 \lea N$ in $K_\lambda$, if:

      \begin{enumerate}
      \item $P \subseteq \Sbs (M_0)$
      \item $I_1, I_2$ are $\subseteq$-maximal sets in

        $$\{I : I \text{ is independent from }M_0 \text{ over } M \text{ in } N \text{ and } a \in I \Rightarrow \tp (a / M_0; N) \in P\}$$
      \item One of $I_1$, $I_2$ is infinite.
      \end{enumerate}

      Then $I_1$ and $I_2$ are both infinite and $|I_1| = |I_2|$.
  \end{enumerate}
\end{cor}
\begin{proof} \
  \begin{enumerate}
    \item $\s_{\ge \lambda}$ is a good $(\ge \lambda)$-frame by Corollary \ref{will-transfer-improved}. $\ts$ is a good $(<\infty, \ge \lambda)$-frame by Corollary \ref{infty-good-frame}.
    \item By Proposition \ref{tameness-uniq}.
    \item By Theorem \ref{full-commut}.
    \item By Theorem \ref{sym-equiv}, Proposition \ref{sym-strong-cont}, and Corollary \ref{sym-elong-infty}.
    \item By Corollary \ref{good-dimension}.
  \end{enumerate}
\end{proof}

\bibliographystyle{amsalpha}
\bibliography{tameness-frames}

\end{document}